\documentclass[12pt]{article}
\usepackage[top=2.5cm,bottom=2.5cm,left=2.5cm,right=2.5cm]{geometry}
 \usepackage{mathptmx}
\usepackage{amssymb}
\usepackage{amsmath,amsthm}
\usepackage[latin1]{inputenc}
\usepackage[dvips]{graphicx}
\usepackage{hyperref}
\usepackage{color}
\usepackage{mathrsfs}
\usepackage{enumerate}
\usepackage{tikz}
\usepackage{xifthen}
\usepackage{verbatim}
\usepackage{soul}

\hypersetup{colorlinks=true}

\hypersetup{colorlinks=true, linkcolor=blue, citecolor=blue,urlcolor=blue}

\newtheorem{remark}{Remark}[section]

\newtheorem{lemma}[remark]{Lemma}
\newtheorem{theorem}[remark]{Theorem}
\newtheorem{proposition}[remark]{Proposition}

\newtheorem{corollary}[remark]{Corollary}

\title{From Italian domination in lexicographic product graphs to w-domination in graphs}
\author{A. Cabrera Mart\'{\i}nez, A. Estrada-Moreno, J. A. Rodr\'{\i}guez-Vel\'{a}zquez\\
{\small Universitat Rovira i Virgili }\\{\small Departament d'Enginyeria Inform\`atica i Matem\`atiques } \\  {\small Av. Pa\"{\i}sos Catalans 26, 43007 Tarragona, Spain.} \\{\small
  abel.cabrera\@@urv.cat, alejandro.estrada\@@urv.cat, juanalberto.rodriguez\@@urv.cat}
}

\date{ }
\begin{document}
\maketitle

\begin{abstract}
In this paper, we show that the Italian domination number of every lexicographic product graph $G\circ H$ can be expressed in terms of five different domination parameters of $G$. These parameters can be defined under the following unified approach, which encompasses the definition of several well-known domination parameters and introduces new ones. 

Let $N(v)$ denote the open neighbourhood of $v\in V(G)$,  
 and let $w=(w_0,w_1, \dots,w_l)$ be a vector of nonnegative integers such that $ w_0\ge 1$. We say that a function $f: V(G)\longrightarrow \{0,1,\dots ,l\}$ is a  $w$-dominating function if  $f(N(v))=\sum_{u\in N(v)}f(u)\ge w_i$ for every vertex $v$ with $f(v)=i$. The weight of $f$ is defined to be $\omega(f)=\sum_{v\in V(G)} f(v)$.
 The  $w$-domination number of $G$, denoted by $\gamma_{w}(G)$, is the minimum weight among all $w$-dominating functions on $G$.  
  
Specifically,  we show that $\gamma_{I}(G\circ H)=\gamma_{w}(G)$, where $w\in \{2\}\times\{0,1,2\}^{l}$ and $l\in \{2,3\}$. The decision on whether  the equality holds for specific values of $w_0,\dots,w_l$
 will depend on the value of the domination number of $H$.
 This paper also provides preliminary results on $\gamma_{w}(G)$ and raises the challenge of conducting a detailed study of the topic. 
\end{abstract}


\section{Introduction}
Let $G$ be a graph, $l$ a positive integer, and $f: V(G)\longrightarrow \{0,\dots , l\}$ a function. For every $i\in \{0,\dots , l\}$, we define $V_i=\{v\in V(G):\; f(v)=i\}$. We will identify $f$ with the subsets $V_0,\dots,V_l$ associated with it, and so we will use the unified notation $f(V_0,\dots , V_l)$ for the function and these associated subsets.  The weight of $f$ is defined to be $$\omega(f)=f(V(G))=\sum_{i=1}^l i|V_i| .$$ 

An \emph{Italian dominating function} (IDF) on a graph $G$   is a
function $f(V_0,V_1,V_2)$ satisfying that $f(N(v))=\sum_{u\in N(v)}f(u)\ge 2$ for every $v\in V_0$, where $N(v)$ denotes the open neighbourhood of $v$.  Hence, $f(V_0,V_1,V_2)$ is an IDF if $N(v)\cap V_2\ne \varnothing$ or $|N(v)\cap V_1|\ge 2$ for every $v\in V_0$.  The \textit{Italian domination number}, denoted by  $\gamma_I(G)$, is the minimum weight among all  IDFs on $G$. This concept was introduced by Chella\-li et al.\ in \cite{CHELLALI201622} under the name of Roman $\{2\}$-domination. The term ``Italian domination" comes from a subsequent paper by Henning and Klostermeyer \cite{HENNING2017557}. 

In this paper we show that the Italian domination number of every lexicographic product graph $G\circ H$ can be expressed in terms of five different domination parameters of $G$. These parameters can be defined under the following unified approach. 

Let $w=(w_0, \dots,w_l)$ be a vector of nonnegative integers such that $ w_0\ge 1$. 
We say that $f(V_0,\dots , V_l)$ is a  $w$-\emph{dominating function} if  $f(N(v))\geq w_i$ for every $v\in V_i$. 
The  $w$-\emph{domination number} of $G$, denoted by $\gamma_{w}(G)$, is the minimum weight among all $w$-dominating functions on $G$. For simplicity, a  $w$-dominating function $f$  of weight $\omega(f)=\gamma_{w}(G)$ will be called a $\gamma_{w}(G)$-function. 

This unified approach allows us to encompass the definition of several well-known domination parameters and introduce new ones. For instance, we would highlight  the following particular cases of known domination parameters that we define here in terms of $w$-domination. 
 
\begin{itemize}
\item The \emph{domination number} of $G$ is defined to be $\gamma(G)=\gamma_{(1,0)}(G)=\gamma_{(1,0,0)}(G)$. Obviously, every $\gamma_{(1,0,0)}(G)$-function $f(V_0,V_1,V_2)$ satisfies that $V_2=\varnothing$ and $V_1$ is a dominating set of cardinality $|V_1|=\gamma(G)$, i.e., $V_1$ is a $\gamma(G)$-set. 

\item The \emph{total domination number} of a graph $G$ with no isolated vertex is defined to be $\gamma_t(G)=\gamma_{(1,1)}(G)=\gamma_{(1,1,w_2, \dots, w_l)}(G)$, for every  $w_2, \dots, w_l\in \{0,1\}$. Notice that there exists a $\gamma_{(1,1,w_2, \dots, w_l)}(G)$-function $f(V_0,V_1,\dots, V_l)$ such that $V_i=\varnothing$ for every $i\in \{2,\dots, l\}$ and $V_1$ is a total dominating set of cardinality $|V_1|=\gamma_t(G)$, i.e., $V_1$ is a $\gamma_t(G)$-set.

\item Given a positive integer $k$, the $k$-\emph{domination number} of a graph $G$  is defined to be
$\gamma_k(G)=\gamma_{(k,0)}(G)$.
In this case, $V_1$ is a $k$-dominating set of cardinality $|V_1|=\gamma_k(G)$, i.e., $V_1$ is a $\gamma_k(G)$-set.
The study of $k$-domination in graphs was initiated by 
Fink and Jacobson
\cite{MR812671} in 1984.

\item  Given a positive integer $k$, the $k$-\emph{tuple domination number} of a graph $G$ of minimum degree $\delta\ge k-1$ is defined to be
$\gamma_{\times k}(G)=\gamma_{(k,k-1)}(G)$.
In this case, $V_1$ is a $k$-tuple dominating set of cardinality $|V_1|=\gamma_{\times k}(G)$, i.e., $V_1$ is a $\gamma_{\times k}(G)$-set.
In particular, $\gamma_{\times 1}(G)=\gamma(G)$ and $\gamma_{\times 2}(G)$ is known as the  \emph{double domination number} of $G$.
This parameter was introduced by Harary and Haynes in \cite{Harary2000}. 

\item  Given a positive integer $k$, the $k$-\emph{tuple total domination number} of a graph $G$ of minimum degree $\delta\ge k$ is defined to be
$\gamma_{\times k,t}(G)=\gamma_{(k,k)}(G)$.
In particular, $\gamma_{\times 1,t}(G)=\gamma_t(G)$ and $\gamma_{\times 2,t}(G)$ is known as the  \emph{double total domination number},   and $V_1$ is a double total dominating set  of cardinality $|V_1|=\gamma_{\times 2,t}(G)$, i.e., $V_1$ is a  $\gamma_{\times 2,t}(G)$-set. 
The $k$-tuple total domination number was introduced by Henning  and Kazemi in \cite{Henning2010}.

\item The \emph{Italian domination number} of $G$ is defined to be $\gamma_{_I}(G)=\gamma_{(2,0,0)}(G)$. As mentioned earlier, this parameter was introduced by Chella\-li et al.\ in \cite{CHELLALI201622} under the name of Roman $\{2\}$-domination number. 
The concept was stu\-died further in \cite{HENNING2017557,Klostermeyer201920}.

\item The \emph{total Italian domination number} of a graph  $G$ with no isolated vertex is defined to be $\gamma_{tI}(G)=\gamma_{(2,1,1)}(G)$. 
This parameter was introduced by Cabrera et al.\ in \cite{TR2DF-2020}, and independently by  Abdollahzadeh~Ahangar et al.\ in \cite{CopionesTotalItalian},  under the name of total Roman $\{2\}$-domination number. The total Italian domination number of lexicographic product graphs was studied in \cite{DD-lexicographic}. 

\item  The \emph{$\{k\}$-domination number} of $G$ is defined to be $\gamma_{\{k\}}(G)=\gamma_{(k,k-1,\dots,1,0)}(G)$.
This parameter was introduced by Domke et al.\ in \cite{k-domination-first} and studied further  in \cite{k-domination-2006,k-domination-2009,k-domination-2008}.
\end{itemize}

Notice that the concept of  $Y$-dominating function introduced by Bange et.al.\  \cite{MR1415278} is quite different from the concept of $w$-dominating function introduced in this paper. Given a set $Y$ of real numbers, a function $ f : V(G) \longrightarrow Y$  is a $Y$-dominating function if $f(N[v])=f(v)+\sum_{u\in N(v)}f(u)\ge 1$
for every $v\in V(G)$. The $Y$-domination number, denoted by $\gamma_{_Y}(G)$,  is the minimum weight among all $Y$-dominating functions on $G$. 
Hence, if $Y=\{0,1,\dots, l\}$, then  $\gamma_{_Y}(G)=\gamma_{(1,0,\dots,0)}(G)=\gamma(G)$.

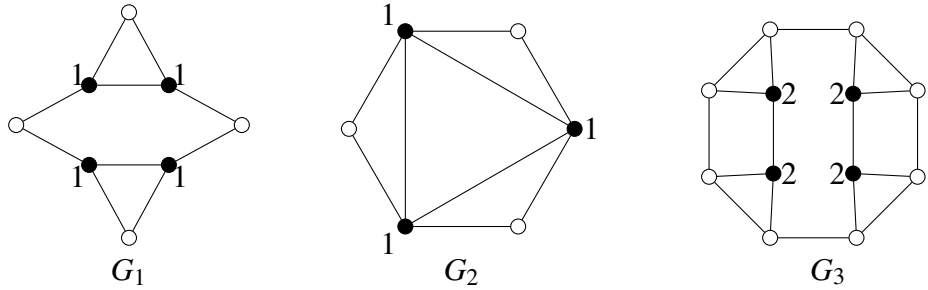
\begin{figure}[ht]
\centering
\begin{tikzpicture}[scale=1.5]
\node[draw, shape=circle, fill=white, scale=.5] at (0:1cm)(A){};
\node[draw, shape=circle, fill=black, scale=.5] at (45:0.5cm)(B){};
\node[draw, shape=circle, fill=white, scale=.5] at (90:1cm)(C){};
\node[draw, shape=circle, fill=black, scale=.5] at (135:0.5cm)(D){};
\node[draw, shape=circle, fill=white, scale=.5] at (180:1cm)(E){};
\node[draw, shape=circle, fill=black, scale=.5] at (225:0.5cm)(F){};
\node[draw, shape=circle, fill=white, scale=.5] at (270:1cm)(G){};
\node[draw, shape=circle, fill=black, scale=.5] at (315:0.5cm)(H){};

\draw (A)--(B)--(C)--(D)--(E)--(F)--(G)--(H)--(A);
\draw (F)--(H); \draw (D)--(B);

\node at ([shift={(.1,.1)}]B) {$1$};
\node at ([shift={(-.1,.1)}]D) {$1$};
\node at ([shift={(-.1,-.1)}]F) {$1$};
\node at ([shift={(.1,-.1)}]H) {$1$};

\node at ([shift={(0,-.3)}]G) {$G_1$};

\end{tikzpicture}
\hspace{1cm}
\begin{tikzpicture}[scale=1.5]
\node[draw, shape=circle, fill=black, scale=.5] at (0:1cm)(A){};
\node[draw, shape=circle, fill=white, scale=.5] at (60:1cm)(B){};
\node[draw, shape=circle, fill=black, scale=.5] at (120:1cm)(C){};
\node[draw, shape=circle, fill=white, scale=.5] at (180:1cm)(D){};
\node[draw, shape=circle, fill=black, scale=.5] at (240:1cm)(E){};
\node[draw, shape=circle, fill=white, scale=.5] at (300:1cm)(F){};

\draw (A)--(B)--(C)--(D)--(E)--(F)--(A)--(C)--(E)--(A);

\node at ([shift={(.15,0)}]A) {$1$};
\node at ([shift={(-.15,.15)}]C) {$1$};
\node at ([shift={(-.15,-.15)}]E) {$1$};
\node at ([shift={(-0.5,-.4)}]F) {$G_2$};
\end{tikzpicture}
\hspace{1cm}
\begin{tikzpicture}[scale=1.5]
\node[draw, shape=circle, fill=white, scale=.5] at (22.5:1cm)(A){};
\node[draw, shape=circle, fill=white, scale=.5] at (22.5+45:1cm)(B){};
\node[draw, shape=circle, fill=white, scale=.5] at (22.5+90:1cm)(C){};
\node[draw, shape=circle, fill=white, scale=.5] at (180-22.5:1cm)(D){};
\node[draw, shape=circle, fill=white, scale=.5] at (180+22.5:1cm)(E){};
\node[draw, shape=circle, fill=white, scale=.5] at (270-22.5:1cm)(F){};
\node[draw, shape=circle, fill=white, scale=.5] at (270+22.5:1cm)(G){};
\node[draw, shape=circle, fill=white, scale=.5] at (360-22.5:1cm)(H){};
\node[draw, shape=circle, fill=black, scale=.5] at (45:0.5cm)(I){};
\node[draw, shape=circle, fill=black, scale=.5] at (135:0.5)(K){};
\node[draw, shape=circle, fill=black, scale=.5] at (225:0.5cm)(L){};
\node[draw, shape=circle, fill=black, scale=.5] at (315:0.5)(J){};

\draw (A)--(B)--(C)--(D)--(E)--(F)--(G)--(H)--(A)--(I)--(B);
\draw (I)--(J)--(H); \draw (G)--(J);
\draw (C)--(K)--(L)--(F); \draw (L)--(E);
\draw (K)--(D);

\node at ([shift={(-.14,0)}]I) {$2$};
\node at ([shift={(.14,0)}]K) {$2$};
\node at ([shift={(.14,0)}]L) {$2$};
\node at ([shift={(-.14,0)}]J) {$2$};
\node at ([shift={(-0.25,-.3)}]G) {$G_3$};
\end{tikzpicture}
\caption{The labels of black-coloured vertices describe a $\gamma_{(2,1,0)}(G_1)$-function, a $\gamma_{(2,2,0)}(G_2)$-function and a $\gamma_{(2,2,2)}(G_3)$-function, respectively.}\label{ZZZ}
\end{figure}

For the graphs shown in Figure \ref{ZZZ} we have the following values. 
\begin{itemize}
\item $\gamma_{I}(G_1)=\gamma_{(2,1,0)}(G_1)=\gamma_{(2,2,0)}(G_1)=4<6=\gamma_{(2,2,1)}(G_1)=\gamma_{(2,2,2)}(G_1).$
\item $\gamma_{I}(G_2)=\gamma_{(2,1,0)}(G_2)=\gamma_{(2,2,0)}(G_2)=\gamma_{(2,2,1)}(G_2)=\gamma_{(2,2,2)}(G_2)=3.$
\item $\gamma_{I}(G_3)=\gamma_{(2,1,0)}(G_3)=6<8=\gamma_{(2,2,0)}(G_3)=\gamma_{(2,2,1)}(G_3)=\gamma_{(2,2,2)}(G_3).$
\end{itemize}

The remainder of the paper is organized as follows. In Section \ref{Sectionlexicographic}  we show that for any graph $G$ with no isolated vertex and any nontrivial graph $H$ with $\gamma(H)\ne 3$ or $\gamma_{_I}(H)\ne 3$, the Italian domination number of $G\circ H$ equals one of the following parameters: $\gamma_{(2,1,0)}(G)$, $\gamma_{(2,2,0)}(G)$, $\gamma_{(2,2,1)}(G)$ or $\gamma_{(2,2,2)}(G)$. The specific value  $\gamma_{_I}(G\circ H)$ takes depends on the value of $\gamma(H)$. For the cases where $\gamma_{_I}(H)= \gamma(H)=3$, we show that $  \gamma_{_I}(G\circ H)= \gamma_{(2,2,2,0)}(G)$. 
Section \ref{NewDomination} is devoted to providing some preliminary results on $w$-domination. 
We first describe some general properties of $\gamma_{w}(G)$ and then dedicate a subsection to each of the specific cases declared of interest in Section \ref{Sectionlexicographic}.

We assume that the reader is familiar with the basic concepts, notation  and terminology of domination in graph. If this is not the case, we suggest the textbooks \cite{Haynes1998a,Haynes1998,Henning2013}.  For the remainder of the paper, definitions will be introduced whenever a concept is needed.

\section{Italian domination in lexicographic product graphs}\label{Sectionlexicographic}

The \emph{lexicographic product} of two graphs $G$ and $H$ is the graph $G \circ H$ whose vertex set is  $V(G \circ H)=  V(G)  \times V(H )$ and $(u,v)(x,y) \in E(G \circ H)$ if and only if $ux \in E(G)$ or $u=x$ and $vy \in E(H)$.

Notice that  for any $u\in V(G)$  the subgraph of $G\circ H$ induced by $\{u\}\times V(H)$ is isomorphic to $H$. For simplicity, we will denote this subgraph by $H_u$. Moreover,
 the neighbourhood of $(x,y)\in V(G)\times V(H)$ will be denoted by $N(x,y)$ instead of $N((x,y))$. Analogously, for any function $f$ on $G\circ H$, the image of $(x,y)$   will be denoted by $f(x,y)$ instead of $f((x,y))$.

\begin{lemma}\label{lem-1} 
For any graph $G$  with no isolated vertex and any nontrivial graph $H$ with $\gamma_{_I}(H)\ne 3 $ or $\gamma(H)\ne 3$,  there exists a $\gamma_{_I}(G\circ H)$-function $f$ satisfying that $f(V(H_u))\leq 2$  for every $u\in V(G)$.
\end{lemma}

\begin{proof}
Given an IDF $f$ on $G\circ H$, we define the set $R_f=\{x\in V(G): \, f(V(H_x))\geq 3\}$. Let $f$ be a $\gamma_{_I}(G\circ H)$-function such that $|R_f|$ is minimum among all $\gamma_{_I}(G\circ H)$-functions. 
Suppose that $|R_f|\geq 1$.  Let $u\in R_f$ such that $f(V(H_u))$ is maximum among all vertices belonging to $R_f$. Suppose that $f(V(H_u))> \gamma_{_I}(H)$. In this case we take a  $\gamma_{_I}(H)$-function $h$ and construct an IDF  $g$  defined on $G\circ H$ as $g(u,y)=h(y)$ for every $y\in V(H)$ and $g(x,y)=f(x,y)$ for  every $x\in V(G)\setminus \{u\}$ and $y\in V(H)$. Obviously, $\omega (g)<\omega(f)$, which is a contradiction. 
Thus,  $3\le f(V(H_u))\le \gamma_{_I}(H_u)=\gamma_{_I}(H)$. Now, we analyse the following two cases.

\vspace{0,2cm}
\noindent
Case 1. $f(V(H_u))\geq 4$.  Let $u'\in N(u)$ and $v\in V(H)$. We define a function $f'$ on $G\circ H$ as $f'(u,v)=f'(u',v)=2$, $f'(u,y)=f(u',y)=0$ for every $y\in V(H)\setminus \{v\},$ and $f'(x,y)=f(x,y)$ for  every $x\in V(G)\setminus \{u,u'\}$ and $y\in V(H)$. Notice that $f'$ is an IDF on $G\circ H$ with $\omega(f')\le \omega(f)$ and $|R_{f'}|<|R_f|$, which is a contradiction.

\vspace{0,2cm}
\noindent
Case 2. $f(V(H_u))=3$. Suppose that $\gamma_{_I}(H)\ne 3 $. Since $\gamma_{_I}(H)\geq 4$, there exist $u'\in N(u)$ and $v\in V(H)$ such that $f(u',v)\ge 1$. Hence, the function $f'$ defined in Case 1 is an IDF on  $G\circ H$ with $\omega(f')\le \omega(f)$ and $|R_{f'}|<|R_f|$, which is again a contradiction. 

Thus, $\gamma_{_I}(H)=3$, and so $\gamma(H)\ne 3$, which implies that $\gamma(H)=2$. Let  $\{v_1,v_2\}$ be a $\gamma(H)$-set.  Let $u'\in N(u)$ and   $v'\in V(H)$ such that $f(u',v')=\max \{f(u',y): \, y\in V(H)\}$. Consider the function $f'$ defined as $f'(u,v_1)=f'(u,v_2)=1$,  $f'(u,y)=0$ for every $y\in V(H)\setminus \{v_1,v_2\}$, $f'(u',v')=\min\{2,f(u',v')+1\}$, $f'(u',y)=0$ for every $y\in V(H)\setminus \{v'\}$, and $f'(x,y)=f(x,y)$ for every $x\in V(G)\setminus \{u,u'\}$ and $y\in V(H)$.  Notice that $f'$ is an IDF on $G\circ H$ with $\omega(f')\le \omega(f)$ and $|R_{f'}|<|R_f|$, which is a contradiction.

 Therefore, $R_f=\varnothing  $, and the result follows.
\end{proof}

\begin{theorem}\label{teo-lower-bound}
The following statements hold for any graph $G$  with no isolated vertex and any nontrivial graph $H$ with $\gamma_{_I}(H)\ne 3$ or $\gamma(H)\ne 3$. 
 
\begin{enumerate}
\item[{\rm (i)}] If $\gamma(H)=1$, then $\gamma_{_I}(G\circ H)=\gamma_{(2,1,0)}(G).$
\item[{\rm (ii)}] If $\gamma_2(H)=\gamma(H)=2$, then $\gamma_{_I}(G\circ H)=\gamma_{(2,2,0)}(G)$.
\item[{\rm (iii)}] If $\gamma_2(H)>\gamma(H)=2$, then $\gamma_{_I}(G\circ H)=\gamma_{(2,2,1)}(G)$.
\item[{\rm (iv)}] If $\gamma(H)\geq 3$, then $\gamma_{_I}(G\circ H)=\gamma_{(2,2,2)}(G).$ 
\end{enumerate}
\end{theorem}

\begin{proof}
Let $f(V_0,V_1,V_2)$ be a $\gamma_{I}(G\circ H)$-function which satisfies Lemma \ref{lem-1}. Let  $f'(X_0,X_1,X_2)$  be the function defined on $G$ by 
$X_1=\{x\in V(G): f(V(H_x))=1\}$ and $X_2=\{x\in V(G): f(V(H_x))=2\}$. Notice that $\gamma_{_I}(G\circ H)=\omega(f)=\omega(f')$. We claim that $f'$ is a   $\gamma_{(w_0,w_1,w_2)}(G)$-function. In order to prove this and find the values of $w_0$, $w_1$ and $w_2$, we differentiate the following three cases. 

\vspace{0,2cm}
\noindent
Case $1.$ $\gamma(H)=1$. 
 Assume that $x\in X_0$. Since $f(V(H_x))=0$, for any $y\in V(H)$ we have that  $f(N(x,y)\setminus V(H_x))\ge 2$. Thus,  $f'(N(x))\ge 2$. Now, assume that $x\in X_1$, and let $(x,y)\in V_1$ be the only vertex in $V(H_x)$ such that $f(x,y)>0$. Since $\gamma(H)=1$, for any $z\in V(H)\setminus \{y\}$, we have that $f(N(x,z)\setminus V(H_x))\ge 1$, which implies that $f'(N(x))\ge 1$. Therefore,  $f'$ is a  $(2,1,0)$-dominating function on $G$  and, as a consequence, $\gamma_{_I}(G\circ H)=\omega(f)=\omega(f')\ge \gamma_{(2,1,0)}(G)$.

Now, for any $\gamma_{(2,1,0)}(G)$-function  $g(W_0,W_1,W_2)$ and any universal vertex $v$ of $H$, the function $g'(W_0',W_1',W_2')$, defined by $W_2'=W_2\times \{v\}$ and $W_1'=W_1\times \{v\}$, is an IDF on $G\circ H$. Therefore, $\gamma_{_I}(G\circ H)\leq \omega(g')=\omega(g)=\gamma_{(2,1,0)}(G)$.

\vspace{0,2cm}
\noindent
Case $2.$ $\gamma(H)=2$.
 As in Case 1 we conclude that $f'(N(x))\ge 2$ for every  $x\in X_0$. Now, assume that $x\in X_1$, and let $(x,y)\in V_1$ be the only vertex in $V(H_x)$ such that $f(x,y)>0$. Since $\gamma(H)=2$, there exists a vertex $z\in V(H)$ such that $(x,z)\in V_0\setminus N(x,y)$. Hence,  $f(N(x,z)\setminus V(H_{x}))\ge 2$,  which implies that $f'(N(x))\ge 2$. Therefore,  $f'$ is a  $(2,2,0)$-dominating function on $G$  and, as a consequence, $\gamma_{_I}(G\circ H)=\omega(f)=\omega(f')\ge \gamma_{(2,2,0)}(G)$.
 
Now, if $\gamma_2(H)>\gamma(H)=2$, then for every $x\in X_2$, there exists $y\in V(H)$ such that $(x,y)\in V_0$ and $f(N(x,y)\cap V(H_x))\le 1$, which implies that $f(N(x,y)\setminus  V(H_x))\ge 1$, and so $f'(N(x))\ge 1$. Hence, $f'$ is a  $(2,2,1)$-dominating function on $G$  and, as a consequence, $\gamma_{_I}(G\circ H)=\omega(f)=\omega(f')\ge \gamma_{(2,2,1)}(G)$.

On the other side, if $\gamma_2(H)=2$, then  for any $\gamma_{(2,2,0)}(G)$-function  $g(W_0,W_1,W_2)$ and any $\gamma_2(H)$-set $S=\{v_1,v_2\}$, the function $g'(W_0',W_1',W_2')$, defined by $W_1'=(W_1\times \{v_1\})\cup (W_2\times S)$ and $W_2'=\varnothing$, is an IDF on $G\circ H$. Therefore, $\gamma_{_I}(G\circ H)\leq \omega(g')=\omega(g)=\gamma_{(2,2,0)}(G)$.
 
 Finally, if $\gamma_2(H)>\gamma(H)=2$  then  we take a $\gamma_{(2,2,1)}(G)$-function  $h(Y_0,Y_1,Y_2)$ and a  $\gamma(H)$-set $S'=\{v_1',v_2'\}$, and construct a function $h'(Y_0',Y_1',Y_2')$ on $G\circ H$ by making $Y_1'=(Y_1\times \{v_1'\})\cup (Y_2\times S')$ and $Y_2'=\varnothing$. Obviously, $h'$ is an IDF on $G\circ H$, and so we can conclude that $\gamma_{_I}(G\circ H)\leq \omega(h')=\omega(h)=\gamma_{(2,2,1)}(G)$.
 
\vspace{0,2cm}
\noindent
Case $3.$ $\gamma(H)\ge 3$.
In this case, for every $x\in V(G)$, there exists  $y\in V(H)$ such that $f(N[(x,y)]\cap V(H_x))=0$. Hence,  $f(N(x,y)\setminus V(H_{x}))\ge 2$,  which implies that $f'(N(x))\ge 2$ for every $x\in V(G)$. Therefore,  $f'$ is a  $(2,2,2)$-dominating function on $G$  and, as a consequence, $\gamma_{_I}(G\circ H)=\omega(f)=\omega(f')\ge \gamma_{(2,2,2)}(G)$.

On the other side, for any $\gamma_{(2,2,2)}(G)$-function  $g(W_0,W_1,W_2)$ and any $v\in V(H)$, the function $g'(W_0',W_1',W_2')$, defined by $W_2'=W_2\times \{v\}$ and $W_1'=W_1\times \{v\}$, is an IDF on $G\circ H$. Hence, $\gamma_{_I}(G\circ H)\leq \omega(g')=\omega(g)=\gamma_{(2,2,2)}(G)$.

According to the three cases above, the result follows.
\end{proof}

The following result considers the case  $\gamma_{_I}(H)= \gamma(H)=3$.  

\begin{theorem}\label{CasoRaro*}
If $H$ is a  graph with $\gamma_{_I}(H)= \gamma(H)=3$, then for any graph $G$,
  $$\gamma_{_I}(G\circ H)=\gamma_{(2,2,2,0)}(G).$$
\end{theorem}

\begin{proof}
Let $f(V_0,V_1,V_2)$ be a $\gamma_{I}(G\circ H)$-function, and $f'(X_0,X_1,X_2,X_3)$ the function defined on $G$ by $X_1=\{x\in V(G): f(V(H_x))=1\}$, $X_2=\{x\in V(G): f(V(H_x))=2\}$ and $X_3=\{x\in V(G): f(V(H_x))\ge  3\}$.
 We claim that $f'$ is a   $(2,2,2,0)$-dominating function on $G$. 

Let $x\in X_0\cup X_1\cup X_2$. Since $f(V(H_x))\le 2$ and $\gamma(H)=3$, there exists $y\in V(H)$ such that 
 $f(N[(x,y)]\cap V(H_x))=0$. Thus,  $f'(N(x))\ge 2$ for every $x\in X_0\cup X_1\cup X_2$, which implies that  $f'$ is a  $(2,2,2,0)$-dominating function on $G$. Therefore, 
$\gamma_{_I}(G\circ H)=\omega(f)\ge \omega(f')\ge \gamma_{(2,2,2,0)}(G)$.

On the other side, let $h(Y_0,Y_1,Y_2,Y_3)$ be a 
$\gamma_{(2,2,2,0)}(G)$-function,  $h_1$  a $\gamma_{_I}(H)$-function  and $v\in V(H)$. We define a function $g$  on $G\circ H$  by  $g(x,v)=h(x)$ for every $x\in V(G)\setminus Y_3$, $g(x,y)=0$ for every  $x\in V(G)\setminus Y_3$ and $y\in V(H)\setminus \{v\}$, and $g(x,y)=h_1(y)$ for every $(x,y)\in Y_3\times V(H)$.  A simple case analysis shows that $g$ is an IDF  on $G\circ H$. Therefore, $\gamma_{_I}(G\circ H)\le \omega(g)=\omega(h)= \gamma_{(2,2,2,0)}(G)$. 
\end{proof}

\begin{figure}[ht]
\centering
\begin{tikzpicture}[scale=2]
\node[draw, shape=circle, fill=black, scale=.5] at (-1,0cm)(A){};
\node[draw, shape=circle, fill=black, scale=.5] at (-.5,0)(B){};
\node[draw, shape=circle, fill=black, scale=.5] at (.5,0)(C){};
\node[draw, shape=circle, fill=black, scale=.5] at (1,0)(D){};
\node[draw, shape=circle, fill=white, scale=.5] at (-1,-.7)(E){};
\node[draw, shape=circle, fill=white, scale=.5] at (-1,.7)(F){};
\node[draw, shape=circle, fill=white, scale=.5] at (0,-.7)(G){};
\node[draw, shape=circle, fill=white, scale=.5] at (0,.7)(H){};
\node[draw, shape=circle, fill=white, scale=.5] at (1,-.7)(I){};
\node[draw, shape=circle, fill=white, scale=.5] at (1,.7)(J){};

\draw (A)--(B)--(H)--(C)--(D)--(J);
\draw (I)--(D); \draw (C)--(G)--(B);
\draw (E)--(A)--(F);

\node at ([shift={(-.15,0)}]A) {$2$};
\node at ([shift={(.15,0)}]B) {$1$};
\node at ([shift={(-.15,0)}]C) {$1$};
\node at ([shift={(.15,0)}]D) {$2$};
\end{tikzpicture}
\hspace{1cm}
\begin{tikzpicture}[scale=2]
\node[draw, shape=circle, fill=black, scale=.5] at (-1,0cm)(A){};
\node[draw, shape=circle, fill=white, scale=.5] at (-.5,0)(B){};
\node[draw, shape=circle, fill=black, scale=.5] at (.5,0)(C){};
\node[draw, shape=circle, fill=black, scale=.5] at (1,0)(D){};
\node[draw, shape=circle, fill=white, scale=.5] at (-1,-.7)(E){};
\node[draw, shape=circle, fill=white, scale=.5] at (-1,.7)(F){};
\node[draw, shape=circle, fill=white, scale=.5] at (0,-.7)(G){};
\node[draw, shape=circle, fill=white, scale=.5] at (0,.7)(H){};
\node[draw, shape=circle, fill=white, scale=.5] at (1,-.7)(I){};
\node[draw, shape=circle, fill=white, scale=.5] at (1,.7)(J){};

\draw (A)--(B)--(H)--(C)--(D)--(J);
\draw (I)--(D); \draw (C)--(G)--(B);
\draw (E)--(A)--(F);

\node at ([shift={(-.15,0)}]A) {$2$};
\node at ([shift={(-.15,0)}]C) {$2$};
\node at ([shift={(.15,0)}]D) {$2$};
\end{tikzpicture}
\caption{This figure shows two 
$\gamma_{(2,2,0)}(G)$-functions on the same graph. The function on the left is also a   
$\gamma_{(2,2,1)}(G)$-function.}\label{FigRaro}
\end{figure}
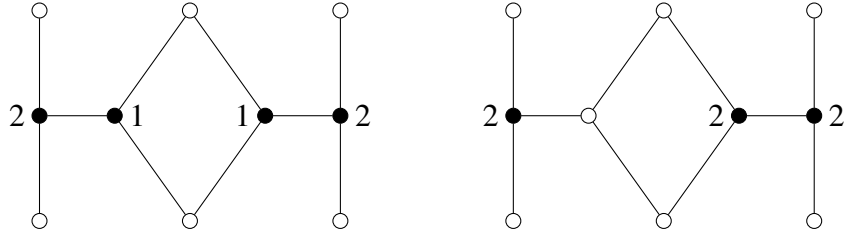

The graph shown in Figure \ref{FigRaro} satisfies  $6=\gamma_{(2,2,0)}(G)=\gamma_{(2,2,1)}(G)<7=\gamma_{(2,2,2,0)}(G)<\gamma_{(2,2,2)}(G)=8$.

\section{Preliminary results on $w$-domination} \label{NewDomination}
In this section, we fix the notation  $\mathbb{Z}^+=\{1,2,3,\dots\}$ and $\mathbb{N}=\mathbb{Z}^+\cup\{0\}$ for the sets of positive and nonnegative integers, respectively. 

Throughout this section, we will repeatedly apply, without explicit mention, the follo\-wing necessary and sufficient condition for the existence of a  $w$-dominating function. 

\begin{remark}
Let $G$ be a graph of minimum degree $\delta$ and let
$w=(w_0, \dots ,w_l)\in \mathbb{Z}^+\times \mathbb{N}^l$. If $ w_0\ge   \cdots \ge w_l$, then  there exists a  $w$-dominating function on $G$ if and only if $w_l\le l\delta .$
\end{remark}

\begin{proof}
Let $w=(w_0, \dots ,w_l)\in \mathbb{Z}^+\times \mathbb{N}^l$ such that $ w_0 \ge \cdots \ge w_l$.  If $w_l\le l\delta $, then the function $f$,  defined by $f(v)=l$ for every $v\in V(G)$,  is a  $w$-dominating function on $G$, as $V_l=V(G)$ and for any $x\in V_l$, $f(N(x))\ge l\delta\ge w_l$.

Now, suppose that $w_l> l\delta $. If  $g$ is  a  $w$-dominating function on $G$, then for any vertex $v$ of degree $\delta$ we have $g(N(v))\le \delta l < w_l\le w_{l-1}\le \dots \le  w_0$, which is a contradiction. Therefore, the result follows.
\end{proof}

 We will show that in general the $w$-domination numbers satisfy a certain monotonicity. Given two integer vectors  $w=(w_0, \dots,w_l)$ and $w'=(w_0', \dots,w_l')$, we say that  $w \prec w'$ if $w_i\le w_i'$ for every $i\in \{0,\dots,l\}$. With this notation in mind, we can state the next remark which is  direct consequence of the definition of $w$-domination number.

\begin{remark}\label{REmarkMonotonicity}
Let $G$ be a graph of minimum degree $\delta$ and let   $w=(w_0, \dots,w_l),w'=(w_0', \dots,w_l')\in \mathbb{Z}^+\times \mathbb{N}^l$ such that $ w_i\ge w_{i+1}$ and  $w_i'\ge w_{i+1}'$ for every $i\in \{0, \dots , l-1\}$ . If $w\prec w'$ and $w_l'\le l\delta$, then every $w'$-dominating function is a $w$-dominating function and, as a consequence, $$\gamma_w(G)\le \gamma_{w'}(G).$$
\end{remark}

We would emphasize the following remark on the specific cases of domination parameters considered in Section \ref{Sectionlexicographic}. Obviously, when we write $\gamma_{(2,2,2)}(G)$ or $\gamma_{(2,2,1)}(G)$, we are assuming that $G$ has minimum degree $\delta\ge 1$.

\begin{remark}\label{REmarkMonotonicityPart}
The following statements hold.
 
\begin{enumerate}
\item[{\rm (i)}]  $\gamma_{_I}(G)=\gamma_{(2,0,0)}(G)\le \gamma_{(2,1,0)}(G)\le \gamma_{(2,2,0)}(G)\le \gamma_{(2,2,1)}(G)\le \gamma_{(2,2,2)}(G).$
\item[{\rm (ii)}] If $w_2\in \{1,2\}$, then $\gamma_{(1,0,w_2)}(G)=\gamma_{(1,0,0)}(G)=\gamma(G)$ and $\gamma_{(1,1,w_2)}(G)=\gamma_{(1,1,0)}(G)=\gamma_{t}(G)$.
\item[{\rm (iii)}] For any integer $k\ge 3$, there exists an infinite  family $\mathcal{H}_k$ of graphs such that for every graph $G\in \mathcal{H}_k$,
$\gamma_{_I}(G)=\gamma_{(2,0,0)}(G)= \gamma_{(2,1,0)}(G)= \gamma_{(2,2,0)}(G)= \gamma_{(2,2,1)}(G)= \gamma_{(2,2,2)}(G)=k.$
\item[{\rm (iv)}] There exists an infinite family of graphs such that
$\gamma_{_I}(G)<\gamma_{(2,1,0)}(G)< \gamma_{(2,2,0)}(G)<\gamma_{(2,2,1)}(G)<\gamma_{(2,2,2)}(G).$
\end{enumerate} 
\end{remark}

In order to see that the remark above holds, we just have to construct families of graphs satisfying (iii) and (iv), as (i) is a particular case of Remark \ref{REmarkMonotonicity} 
and (ii) is derived from the definition of $(w_0,w_1,w_2)$-domination number. In the case of (iii), we  construct a family $\mathcal{H}_k=\{G_{k,r}: \, r\in \mathbb{Z}^+ \}$ as follows. Let $k\ge 3$ be an integer, and let $N_{r}$ be the empty graph  of order $r$. For any positive integer $r$ we construct a graph $G_{k,r}\in \mathcal{H}_k$ from a complete graph $K_k$ and $\binom{k}{2}$ copies of $N_r$, in such  way that for each pair of different vertices $\{x,y\}$ of $K_k$ we choose one copy of $N_r$ and  connect  every vertex of $N_r$ with $x$ and $y$, making $x$ and $y$ vertices of degree $(k-1)(r+1)$ in $G_{k,r}$. 
For instance, the graph $G_{3,1}$ is isomorphic to the graph $G_2$ shown in Figure \ref{ZZZ}. It is readily seen that $\gamma_{_I}(G_{k,r})=\gamma_{(2,2,2)}(G_{k,r})=k$. 
On the other hand, in the case of  (iv), we consider the family of cycles of order $n\geq 10$ with $n\equiv 1 \pmod 3$. For these graphs we have that 
 $\gamma_I(C_n)<\gamma_{(2,1,0)}(C_n)<\gamma_{(2,2,0)}(C_n)<\gamma_{(2,2,1)}(C_n)<\gamma_{(2,2,2)}(C_n).$ 
The specific values of $\gamma_{(w_0,w_1,w_2)}(C_n)$ will be given in Subsections \ref{SubSection(2,2,2)}, \dots ,\ref{SubSection(2,1,0)}.

Next we show a class of graphs where 
$ \gamma_{(w_0,\dots ,w_l)}(G)=w_0\gamma(G)$ whenever $l\ge w_0\ge  \cdots \ge w_l$. To this end, we need to introduce some additional notation and terminology.
Given two graphs $G_1$ and $G_2$, the \emph{corona product graph} $G_1\odot G_2$ is the graph obtained from $G_1$ and $G_2$, by taking one copy of $G_1$ and $|V(G_1)|$ copies of $G_2$ and joining by an edge every vertex from the $i^{th}$-copy of $G_2$ with the $i^{th}$-vertex of $G_1$. For every $x\in V(G_1)$, the copy of $G_2$ in $G_1\odot G_2$ associated to $x$ will be denoted by $G_{2,x}$.  It is well known that $\gamma(G_1\odot G_2)=|V(G_1)|$ and, if $G_1$ does not have isolated vertices, then $\gamma_t(G_1\odot G_2)=\gamma(G_1\odot G_2)=|V(G_1)|$.

\begin{theorem}\label{IgualdadCorona}
Let $G\cong G_1\odot G_2$ be a corona graph where $G_1$ does not have isolated vertices, and let $w=(w_0,\dots ,w_l)\in \mathbb{Z}^+\times \mathbb{N}^l$. If $l\ge w_0\ge  \cdots \ge w_l$ and $|V(G_2)|\ge w_0$, then
$$\gamma_{w}(G)=w_0\gamma(G).$$
\end{theorem}

\begin{proof}
Since $G_1$ does not have isolated vertices, the upper bound $\gamma_{w}(G)\le w_0|V(G_1)|=w_0\gamma(G)$ is straightforward, as the function $f$, defined by $f(x)=w_0$ for every vertex $x\in V(G_1)$ and $f(x)=0$ for every $x\in V(G)\setminus V(G_1)$, is a $w$-dominating function on $G$.

On the other hand, let $f$ be a $ \gamma_{w}(G)$-function and suppose that there exists $x\in V(G_1)$ such that $f(V(G_{2,x}))+f(x)\le w_0-1$. In such a case, $f(N[y])\le w_0-1$ for every $y\in V(G_{2,x})$,  which is a contradiction, as $|V(G_2)|\ge w_0$.   Therefore,  $\gamma_{w}(G)= \omega(f)\ge  w_0|V(G_1)|=w_0 \gamma(G)$.
\end{proof}

\begin{proposition}\label{obs-subgraph-general}
Let $G$ be a graph of order $n$. Let
$w=(w_0,\dots ,w_l)\in \mathbb{Z}^+\times \mathbb{N}^l$ such that $ w_0\ge  \cdots \ge w_l$.
If $G'$ is a spanning subgraph of $G$ with 
 minimum degree $\delta'\ge \frac{w_l }{l}$,   then $$\gamma_{w}(G)\leq \gamma_{w}(G').$$
\end{proposition}

\begin{proof}
Let $E^-=\{e_1,\dots, e_k\}$ be the set of all edges of $G$ not belonging to the edge set of $G'$. Let $G'_0=G$ and, for every $i\in \{1,\dots, k\}$, let  $X_i=\{e_1,\dots , e_i\}$ and $G'_i=G-X_i$, the edge-deletion subgraph of $G$ induced by  $E(G)\setminus X_i$.  Since any $w$-dominating function on  $G'_i$ is a  $w$-dominating function on  $G'_{i-1}$, we can conclude that $\gamma_{w}(G'_{i-1})\leq \gamma_{w}(G'_{i})$. Hence,  $\gamma_{w}(G)=\gamma_{w}(G'_0)\leq \gamma_{w}(G'_1)\le \cdots \le \gamma_{w}(G'_k)= \gamma_{w}(G')$.
\end{proof}

From Proposition \ref{obs-subgraph-general}  we obtain the following result.

\begin{corollary}\label{corollaryInducedSubgraph}
Let $G$ be a graph of order $n$ and $w=(w_0,\dots ,w_l)\in \mathbb{Z}^+\times \mathbb{N}^l$ such that $ w_0\ge  \cdots \ge w_l$.

\begin{itemize}
\item If $G$ is a Hamiltonian graph and $w_l\le 2l$, then $\gamma_{w}(G)\leq \gamma_{w}(C_n)$.

\item If $G$ has a Hamiltonian path and $w_l\le l$,  then $\gamma_{w}(G)\leq \gamma_{w}(P_n)$.
\end{itemize} 
\end{corollary}

In order to derive lower bounds on the $w$-domination number, we need to state the following useful lemma.
\begin{lemma}\label{LemmaCotasINf}
Let $G$ be a graph with no isolated vertex, maximum degree $\Delta$ and order $n$. For any $w$-dominating function $f(V_0,\dots,V_l)$ on $G$ such that $ w_0\ge  \cdots \ge w_l$, 
$$\Delta\omega(f)\ge w_0n + \sum_{i=1}^l(w_i-w_0)|V_i|.$$
\end{lemma}

\begin{proof}
The result follows from the simple fact that the contribution of any vertex $x\in V(G)$ to the sum $\displaystyle\sum_{x\in V(G)}f(N(x))$ equals $\deg(x)f(x)$, where $\deg(x)$ denotes the degree of $x$. Hence,
\begin{align*}
\Delta\omega(f)&=
\Delta\sum_{x\in V(G)}f(x)\\
&\ge \sum_{x\in V(G)}\deg(x)f(x)\\
& = \sum_{x\in V(G)}f(N(x))\\
&\ge w_0|V_0|+\sum_{i=1}^l w_i|V_i|\\
&= w_0n + \sum_{i=1}^l(w_i-w_0)|V_i|.
\end{align*}
Therefore, the result follows.
\end{proof}

\begin{corollary}\label{CorollaryLowerBounds}
The following statements hold for $k, l\in \mathbb{Z}^+$ and a graph $G$  with  minimum degree $\delta\ge 1$, maximum degree $\Delta$ and order $n$.
\begin{enumerate}
\item[{\rm (i)}] If $k\le l \delta+1$ and $w=({\small \underbrace{k+l-1,k+l-2, \dots ,k-1}_{l+1}})$, then 
      $\gamma_{w}(G)\ge \left\lceil  \frac{(k+l-1)n}{\Delta+1} \right\rceil$     
\item[{\rm (ii)}]  If $k\le l \delta$ and $w=({\small \underbrace{k,\dots ,k}_{l+1}})$, then 
      $\gamma_{w}(G)\ge \left\lceil  \frac{kn}{\Delta} \right\rceil$      
\item[{\rm (iii)}] If $k\le l \delta+1$ and $w=({\small \underbrace{k,k-1,\dots ,k-1}_{l+1}})$, then 
      $\gamma_{w}(G)\ge \left\lceil  \frac{kn}{\Delta+1} \right\rceil$
\item[{\rm (iv)}] Let $w=(w_0,\dots ,w_l)$ with $w_0\ge \cdots \ge w_l$. If $l\delta \ge w_l$,  then 
      $\gamma_{w}(G)\ge  \left\lceil  \frac{w_0n}{\Delta+w_0} \right\rceil$.
\end{enumerate}
\end{corollary}

In the next subsections we shall show that lower bounds above are tight. 
Corollary \ref{CorollaryLowerBounds} implies the following known bounds.
$$
\gamma(G)\ge \left\lceil  \frac{n}{\Delta+1} \right\rceil, \quad \gamma_t(G)\ge \left\lceil  \frac{n}{\Delta} \right\rceil, \quad \gamma_{_I}(G)\ge \left\lceil  \frac{2n}{\Delta+2} \right\rceil, \quad \gamma_{tI}(G)\ge \left\lceil  \frac{2n}{\Delta+1} \right\rceil, $$
$$\gamma_{k}(G)\ge  \left\lceil  \frac{kn}{\Delta+k} \right\rceil, \quad \gamma_{\times k}(G)\ge \left\lceil  \frac{kn}{\Delta+1} \right\rceil, \text{  }\gamma_{\{k\}}\ge \left\lceil\dfrac{kn}{\Delta+1}\right\rceil \text{  } \text{ and } \text{  }  \gamma_{\times k, t}(G)\ge \left\lceil  \frac{kn}{\Delta} \right\rceil.
$$ 

It is readily seen that
$\gamma_{(w_0,\dots ,w_l)}(G)=1 $ if and only if $w_0=1$, $w_1=0$ and $\gamma(G)=1$. Next we characterize the graphs with $\gamma_{(w_0,\dots ,w_l)}(G)=2.$

 \begin{theorem}\label{teo-case=2}
 Let $w=(w_0,\dots ,w_l)\in \mathbb{Z}^+\times \mathbb{N}^l$ such that $ w_0\ge  \cdots \ge w_l$.
For a  graph $G$ of order at least three, $\gamma_{(w_0,\dots ,w_l)}(G)=2$ if and only if one of the following conditions holds.
\begin{enumerate}
\item[{\rm (i)}] $w_2=0$, $\gamma(G)=1$ and either $w_0=2$ or $w_0=w_1=1$.
\item[{\rm (ii)}] $w_0=1$, $w_1=0$ and $\gamma(G)=2$.
\item[{\rm (iii)}] $w_0=1$, $w_1=1$ and $\gamma_t(G)=2$.
\item[{\rm (iv)}] $w_0=2$, $w_1=0$ and $\gamma_2(G)=2$.
\item[{\rm (v)}] $w_0=2$, $w_1=1$ and $\gamma_{\times 2}(G)=2$.
\end{enumerate}
\end{theorem}
 
\begin{proof}
Assume first that $\gamma_{(w_0,\dots ,w_l)}(G)=2$ and let $f(V_0,\dots ,V_l)$ be a $\gamma_{(w_0,\dots ,w_l)}(G)$-function. Notice that $w_0\in \{1,2\}$ and  $|V_2|\in \{0,1\}$.  
If $|V_2|=1$, then $w_2=0$ and $V_i=\varnothing$ for every $i\ne 0,2$.  Hence, $\gamma(G)=1$ and either $w_0=2$ or $w_0=w_1=1$. Therefore,  (i) follows.

Now we consider the case $V_2=\varnothing$. Notice that $V_1$ is a dominating set of cardinality two,  $w_1\in \{0,1\}$  and $V_i=\varnothing$ for every $i\ne 0,1$.

Assume first that $w_0=1$ and $w_1=0$. If $\gamma(G)=1$, then $\gamma_{(w_0,\dots ,w_l)}(G)=1$, which is a contradiction. Hence, $\gamma(G)=2$ and so (ii) follows.
For $w_0+w_1\ge 2$ we have the following possibilities.

If $w_0=w_1=1$, then $V_1$ is a total dominating set of cardinality two, and so $\gamma_t(G)=2$. Therefore, (iii) follows.

If $w_0=2$ and $w_1=0$, then $V_1$ is a $2$-dominating set of cardinality two, which implies that $\gamma_2(G)=2$. Therefore, (iv) follows.

If $w_0=2$ and $w_1=1$, 
then $V_1$ is a double dominating set of cardinality two, and this implies that 
$\gamma_{\times 2}(G)=2$. Therefore, (v) follows.

Conversely, if one of the five conditions holds, then it is easy to check that $\gamma_{(w_0,\dots ,w_l)}(G)=2$, which completes the proof. 
\end{proof}

In order to establish the following result,  we need to define the following parameter. 
 $$\nu_{(w_0, \dots ,w_l)}(G)=\max \{|V_0|:\, f(V_0,\dots, V_l) \text{ is a } \gamma_{(w_0,\dots ,w_l)}(G)\text{-function}.\}$$
 In particular, for $l=1$ and a graph  $G$ of  order $n$, we have that  $\nu_{(w_0, w_1)}(G)=n-\gamma_{(w_0,w_1)}(G)$.

\begin{theorem}\label{GeneralUpperBounds(2,2,2)}
Let $G$ be a graph  of minimum degree $\delta$ and order $n$. The following statements hold for any  $(w_0,\dots ,w_l)\in \mathbb{Z}^+\times \mathbb{N}^l$ with $w_0\ge \cdots \ge w_l$.
\begin{enumerate}
\item[{\rm (i)}] If there exists $i\in \{1,\dots, l-1\}$ such that $i\delta\geq w_i$,  then
$$\gamma_{(w_0,\dots ,w_l)}(G)\leq \gamma_{(w_0,\dots ,w_i)}(G).$$ 
\item[{\rm (ii)}] If $l\ge i+1\ge w_0$, then $$\gamma_{(w_0,\dots ,w_i,0,\dots,0)}(G)\le (i+1)\gamma(G).$$
\item[{\rm (iii)}] Let $k,i\in \mathbb{Z}^+$ such that $l\ge ki$, 
and let $(w'_0, w_1',\dots ,w_i')\in \mathbb{Z}^+\times \mathbb{N}^l$. 
 If   $i\delta\geq w_i'$ and $w_{kj}=kw_{j}'$ for every $j\in \{0,1,\dots, i\}$, then $$\gamma_{(w_0,\dots ,w_l )}(G)\le k \gamma_{(w'_0,\dots ,w'_i)}(G).$$

\item[{\rm (iv)}]  Let $k\in \mathbb{Z}^+$ and $\beta_1, \dots ,\beta_k\in \mathbb{Z}^+$. If $l\delta\geq k+w_l>k$ and $w_0+k\ge \beta_1\ge  \cdots \ge \beta_k\ge w_1+k$,
then  $$\gamma_{(w_0+k,\beta_1, \dots ,\beta_k,w_1+k,\dots ,w_l+k)}(G)\le \gamma_{(w_0,\dots ,w_l)}(G)+k(n-\nu_{(w_0, \dots ,w_l)}(G)).$$

\item[{\rm (v)}] If $l\delta\ge w_l\ge l\ge 2$, then $$\gamma_{(w_0,\dots ,w_l)}(G)\leq l \gamma_{(w_0-l+1,w_l-l+1)}(G).$$

\item[{\rm (vi)}] If  $\delta\ge 1$, $w_0\le l-1$ and $w_{l-1}\ge 1$, then   $$\gamma_{(w_0,\dots,w_{l-2},1)}(G)\le\gamma_{(w_0,\dots,w_{l-1},0)}(G).$$
\end{enumerate} 
\end{theorem}

\begin{proof}
If there exists $i\in \{1,\dots, l-1\}$ such that $i\delta\geq w_i$, then   for any $\gamma_{(w_0,\dots ,w_i)}(G)$-function $f(V_0,\dots, V_i)$ we define a $(w_0,\dots ,w_l)$-dominating function $g(W_0,\dots , W_l)$ by $W_j=V_j$ for every $j\in \{0,\dots , i\}$ and $W_j=\varnothing$ for every $j\in \{i+1,\dots , l\}$. Hence, 
$\gamma_{(w_0,\dots ,w_l)}(G)\le \omega(g)=\omega(f)=\gamma_{(w_0,\dots ,w_i)}(G)$.  Therefore, (i) follows.

Now, assume $l\ge i+1\ge w_0$. Let $S$ be a $\gamma(G)$-set. Let $f$ be the  function  defined by $f(v)=i+1$ for every $v\in S$ and $f(v)=0$ for the remaining vertices. Since $f$ is 
 a  $(w_0,\dots ,w_i, 0\dots, 0)$-dominating   function, we conclude that 
$\gamma_{(w_0,\dots ,w_i, 0\dots, 0)}(G)\le \omega(f)=(i+1)|S|=(i+1)\gamma(G)$, which implies that (ii) follows.

In order to prove (iii), assume that $l\ge ki$, $i\delta\geq w_i'$ and $w_{kj}=kw'_j$ for every $j\in \{0,\dots, i\}$. Let $f'(V_0',\dots, V_i')$ be a $\gamma_{(w_0',\dots ,w_i')}(G)$-function. We construct a function $f(V_0,\dots, V_l)$ as $f(v)=kf'(v)$ for every $v\in V(G)$. Hence, $V_{kj}=V_j'$ for every $j\in \{0,\dots, i\}$, while $V_j=\varnothing$ for the remaining cases. Thus,  
for every $v\in V_{kj}$ with $j\in \{0,\dots, i\}$ we have that $f(N(v))=kf'(N(v))\ge kw_j'= w_{kj}$, which implies that $f$ is a  $(w_0,\dots ,w_l)$-dominating function, and so $\gamma_{(w_0,\dots ,w_l)}(G)\le \omega(f)=k\omega(f')=k\gamma_{(w_0',\dots ,w_i')}(G)$. Therefore,   (iii) follows.

Now, assume that $l\delta\geq k+w_l>k$ and $w_0+k\ge \beta_1\ge  \cdots \ge \beta_k\ge w_1+k$.
Let $g(W_0,\dots, W_l)$ be a $\gamma_{(w_0,\dots ,w_l)}(G)$-function. We construct a function $f(V_{0},\dots, V_{l+k})$ as $f(v)=g(v)+k$ for every $v\in V(G)\setminus W_0$ and $f(v)=0$ for every  $v\in W_0$. Hence, $V_{j+k}=W_j$ for every $j\in \{1,\dots, l\}$, $V_0=W_0$ and $V_j=\varnothing$ for the remaining cases. Thus,  
if $v\in V_{j+k}$ and $j\in \{1,\dots, l\}$, then  $f(N(v))\ge g(N(v))+k\ge w_j+k$,  and if $v\in V_{0}$, then  $f(N(v))\ge g(N(v))+k\ge w_0+k$. This implies that $f$ is a  $(w_0+k,\beta_1, \dots ,\beta_k,w_1+k,\dots ,w_l+k)$-dominating function, and so $\gamma_{(w_0+k,\beta_1, \dots ,\beta_k,w_1+k,\dots ,w_l+k)}(G)\le \omega(f)=\omega(g)+ k\sum_{j=1}^l|W_j|=\gamma_{(w_0,\dots ,w_l)}(G)+k(n-|W_0|)\le \gamma_{(w_0,\dots ,w_l)}(G)+k(n-\nu_{(w_0, \dots ,w_l)}(G))$. Therefore,   (iv) follows. 

Furthermore, if $l\delta\ge w_l\ge l\ge 2$, then by applying (iv) for $k=l-1$, we deduce that 
$$\gamma_{(w_0,\dots ,w_l)}(G)\leq  \gamma_{(w_0-l+1,w_l-l+1)}(G)+(l-1)(n-\nu_{(w_0-l+1,w_l-l+1)}(G))=l\gamma_{(w_0-l+1,w_l-l+1)}(G).$$ 
Therefore,   (v) follows.

From now on, let $\delta\ge 1$, $w_0\le l-1$ and $w_{l-1}\ge 1$. Let $f(V_0,\ldots, V_l)$ be a $\gamma_{(w_0,\dots,w_{l-1},0)}(G)$-function. Assume first $V_l=\varnothing  $. Since $w_{l-1}\ge 1$,  we have that $f$ is a  $(w_0,\dots,w_{l-2},1)$-dominating function on $G$, which implies that  (vi) follows.
Assume now that there exists $v\in V_l$. If $f(N(v))\ge l-1$, then 
 the function $f'$, defined by $f'(v)=l-1$ and  $f'(x)=f(x)$ for every $x\in V(G)\setminus\{v\}$, is a  $(w_0,\dots,w_{l-1},0)$-dominating function with $\omega(f')<\omega(f)$, which is a contradiction. Hence, $f(N(v))\le l-2$ for every $v\in V_l$.  
 Since  $\delta\ge 1$, for each vertex $x\in V_l$, we fix one vertex $x'\in N(x)$ and we form a set $S$ from them such that $|S|\le |V_l|$. Let $g$ be the function  defined by $g(x)=f(x)+1$ for any $x\in S$, $g(y)=l-1$ for any $y\in V_l$, and $g(z)=f(z)$ for the remaining vertices of $G$. Since $g(N(x))\ge l-1\ge w_i$ for every $x\in S$ and $i\in \{0,\dots,l-2\}$, $g(N(y))\ge 1$ for every $y\in V_{l-1}\cup V_l$, and $g(N(z))\ge w_i$ for 
 every $z\in V_i\setminus (S\cup V_{l-1}\cup V_l)$ and $i\in \{0,\dots,l-2\}$, we conclude that $g$ 
  is a  $(w_0,\dots,w_{l-2},1)$-dominating function on $G$. Therefore, $\gamma_{(w_0,\dots,w_{l-2},1)}(G)\le\omega(g)\le\omega(f)=\gamma_{(w_0,\dots,w_{l-1},0)}(G)$, which completes the proof of  (vi).  
\end{proof}


In the next subsections we consider several applications of Theorem \ref{GeneralUpperBounds(2,2,2)} where we show that the bounds are tight.
For instance, the following particular cases will be of interest. 


\begin{corollary}\label{CorollaryGeneralUpperBounds}
Let $G$ be a graph of minimum degree $\delta$, and let $k,l,w_2,\dots ,w_l\in \mathbb{Z}^+$ with $k\ge w_2 \ge \cdots \ge w_l$.
\begin{enumerate}

\item[{\rm (i)}]  If $\delta \ge k$ and $w=(k+1,k , w_2,\dots ,w_l)$, then $\gamma_{w}(G)\le \gamma_{\times k}(G)$.

\item[{\rm (ii)}]  If $\delta \ge k$ and $w=(k, k, w_2,\dots ,w_l)$, then $\gamma_{w}(G)\le \gamma_{\times k,t}(G)$.

\item[{\rm (iii)}]  If $l\delta \ge k\ge l\ge 2$ and $w= ({\scriptsize  \underbrace{k+1,k, \dots ,k}_{l+1}})$, then $\gamma_{w}(G)\le l\gamma_{\times (k-l+2)}(G)$.

\item[{\rm (iv)}]  If $l\delta \ge k\ge l\ge 2$ and $w=({\scriptsize \underbrace{k,k, \dots ,k}_{l+1}})$, then $\gamma_{w}(G)\le l\gamma_{\times (k-l+1),t}(G)$.

\item[{\rm (v)}]  If $l\ge k$,  $\delta\ge 1$ and $w=({\scriptsize\underbrace{k, \dots ,k}_{l+1}})$, then $\gamma_{w}(G)\le k\gamma_{t}(G)$.
\end{enumerate} 
\end{corollary}

\begin{proof}
If $\delta \ge k$, then by Theorem \ref{GeneralUpperBounds(2,2,2)} (i) we conclude that (i) and (ii) follows.

If $l\delta \ge k\ge l\ge 2$, then by Theorem \ref{GeneralUpperBounds(2,2,2)} (v) we deduce  that
$$\gamma_{({\scriptsize \underbrace{k+1,k, \dots ,k}_{l+1}})}(G)\le l 
\gamma_{(k-l+2,k-l+1)}(G)=
l\gamma_{\times (k-l+2)}(G).$$
 Hence, (iii) follows. By analogy we derive (iv), as $\gamma_{(k-l+1,k-l+1)}(G)=
l\gamma_{\times (k-l+1),t}(G)$.

Finally, if $l\ge k$ and $\delta\ge 1$, then 
by Theorem \ref{GeneralUpperBounds(2,2,2)} (iii) we deduce  that
$$\gamma_{({\scriptsize \underbrace{k, \dots ,k}_{l+1}})}(G)\le k\gamma_{(1,1)}(G)=k\gamma_{t}(G).$$
 Therefore, (v) follows. 
\end{proof}

\subsection{Preliminary results on $(2,2,2)$-domination}
\label{SubSection(2,2,2)}

\begin{theorem}\label{general_bound(2,2,2)}
For any graph $G$ with no isolated vertex, order $n$ and maximum degree $\Delta$,
$$\left\lceil  \frac{2n}{\Delta} \right\rceil\le \gamma_{(2,2,2)}(G)\le 2\gamma_t(G).$$
Furthermore, if $G$ has minimum degree $\delta\geq 2$, then
$$\gamma_{(2,2,2)}(G)\leq \gamma_{\times 2,t}(G).$$
\end{theorem}

\begin{proof}
From Corollary \ref{CorollaryLowerBounds} we deduce   the lower bound. 
The upper bound $\gamma_{(2,2,2)}(G)\le 2\gamma_t(G)$  follows by  Corollary \ref{CorollaryGeneralUpperBounds} (v), while, if $\delta\ge 2$, then we apply Corollary \ref{CorollaryGeneralUpperBounds} (ii) to deduce that  $\gamma_{(2,2,2)}(G)\le \gamma_{\times 2,t}(G)$. Therefore, the result follows.
\end{proof}


The bounds above are tight. For instance, for the graphs $G_2$ and $G_3$ shown in   Figure~\ref{ZZZ} we have that $\left\lceil  \frac{2n }{\Delta } \right\rceil= \gamma_{(2,2,2)}(G_2)=\gamma_{\times 2,t}(G_2)=3$ and $\gamma_{(2,2,2)}(G_3)= 2\gamma_t(G_3)=8.$ Notice that every graph $G_{k,r}$ belonging to the infinite family $\mathcal{H}_k$ constructed after Remark \ref{REmarkMonotonicityPart} satisfies the equality  $\gamma_{(2,2,2)}(G_{k,r})=\gamma_{\times 2,t}(G_{k,r})=k$.
Furthermore, from Theorem \ref{IgualdadCorona}  we have that for any corona graph $G\cong G_1\odot G_2$, where $G_1$ does not have isolated vertices, $\gamma_{(2,2,2)}(G)=2\gamma(G)=2\gamma_t(G)$.


Notice that by theorem \ref{general_bound(2,2,2)} we have that $ \gamma_{(2,2,2)}(G)\ge \left\lceil  \frac{2n}{\Delta} \right\rceil\ge 3$ for every graph $G$ with no isolated vertex. Next we characterize all graphs with $\gamma_{(2,2,2)}(G)=3$. To this end, we need to establish the following lemma.

\begin{lemma}\label{lem-small-values-(2,2,2)}
For a graph $G$, the following statements are equivalent.

\begin{enumerate} 
\item[{\rm (i)}] $\gamma_{(2,2,2)}(G)=\gamma_{\times 2,t}(G)$.
\item[{\rm (ii)}] There exists a $\gamma_{(2,2,2)}(G)$-function $f(V_0,V_1,V_2)$ such that $V_2=\varnothing$.
\end{enumerate}
\end{lemma}

\begin{proof}  
If $\gamma_{(2,2,2)}(G)=\gamma_{\times 2,t}(G)$, then 
for any $\gamma_{\times 2,t}(G)$-set $D$,  the function $g(W_0,W_1,W_2)$, defined by $W_1=D$ and $W_0=V(G)\setminus D$, is a $\gamma_{(2,2,2)}(G)$-function. Therefore, (ii) follows.

Conversely, if there exists a $\gamma_{(2,2,2)}(G)$-function $f(V_0,V_1,V_2)$ such that $V_2=\varnothing$, then $V_1$ is a double total dominating set of $G$, and so $\gamma_{\times 2,t}(G)\leq |V_1|=\omega(f)=\gamma_{(2,2,2)}(G)$. Therefore, Theorem \ref{general_bound(2,2,2)} leads to  $\gamma_{(2,2,2)}(G)=\gamma_{\times 2,t}(G)$.
\end{proof}

\begin{theorem}\label{teo-case=3-(2,2,2)}
For a graph $G$,  the following statements are equivalent.
\begin{enumerate} 
\item[{\rm (i)}] $\gamma_{(2,2,2)}(G)=3$.
\item[{\rm (ii)}] $\gamma_{\times 2,t}(G)=3$.
\end{enumerate}
\end{theorem}

\begin{proof}
 Assume first that $\gamma_{(2,2,2)}(G)=3$, and let $f(V_0,V_1,V_2)$ be a $\gamma_{(2,2,2)}(G)$-function. Suppose that there exists $u\in V_2$. Since $f(N(u))\ge 2$, we deduce that $\gamma_{(2,2,2)}(G)\ge 4$, which is a contradiction. Hence, $V_2=\varnothing$ and  by Lemma \ref{lem-small-values-(2,2,2)} we conclude that  $\gamma_{\times 2,t}(G)=3$. 

Conversely, if $\gamma_{\times 2,t}(G)=3$, then $G$ has minimum degree $\delta\ge 2$ and so Theorem \ref{general_bound(2,2,2)} leads to $3\leq \left\lceil  \frac{2n}{\Delta} \right\rceil\leq \gamma_{(2,2,2)}(G)\leq \gamma_{\times 2,t}(G)=3$. Therefore, $\gamma_{(2,2,2)}(G)=3$. 
\end{proof}

Next we consider the case of graphs with $\gamma_{(2,2,2)}(G)=4$. 

\begin{theorem}\label{teo-case=4-(2,2,2)}
For a graph $G$,  $\gamma_{(2,2,2)}(G)=4$  if and only if at least 
one of the following conditions holds.
\begin{enumerate}
\item[{\rm (i)}] $\gamma_{\times 2,t}(G)=4$.
\item[{\rm (ii)}] $\gamma_t(G)=2$ and $G$ has minimum degree $\delta=1$. 
\item[{\rm (iii)}] $\gamma_t(G)=2$ and  $\gamma_{\times 2,t}(G)\geq 4$.
\end{enumerate}
\end{theorem}

\begin{proof}
Assume $\gamma_{(2,2,2)}(G)=4$. Notice that $G$ does not have isolated vertices. Let $f(V_0,V_1,V_2)$ be a $\gamma_{(2,2,2)}(G)$-function. If $V_2=\varnothing$, then by Lemma \ref{lem-small-values-(2,2,2)} we obtain that $\gamma_{\times 2,t}(G)=\gamma_{(2,2,2)}(G)=4$, and so (i) follows. 

From now on, assume that $|V_2|\in \{1,2\}$. If $|V_2|=2$, then $V_1=\varnothing$ and, as a result, $V_2$ is a total dominating set of $G$, which implies that $\gamma_t(G)=2$. On the other side, if  $|V_2|=1$, then $|V_1|=2$ and both vertices belonging to $V_1$ are adjacent to the vertex of weight two, and every $v\in V_0$ satisfies  $N(v)\cap V_2\neq \varnothing$ or $V_1\subseteq N(v)$. This implies that the union of $V_2$ with a singleton subset of $V_1$  forms a total dominating set of $G$, and again $\gamma_t(G)=2$. Now, if $\delta\ge 2$, then Theorem \ref{general_bound(2,2,2)} leads to $4= \gamma_{(2,2,2)}(G)\le \gamma_{\times 2,t}(G)$. Hence, by Theorem \ref{teo-case=3-(2,2,2)} we conclude that either $\delta=1$ or $\gamma_{\times 2,t}(G)\ge 4$. Therefore, either (ii) or (iii) holds.

Conversely, if $\gamma_{\times 2,t}(G)=4$, then $G$ has minimum degree $\delta\ge 2$ and by Theorem \ref{general_bound(2,2,2)} we have that $3\le \gamma_{(2,2,2)}(G)\le 4$. Hence, by Theorem \ref{teo-case=3-(2,2,2)} we deduce that  $\gamma_{(2,2,2)}(G)=4$.
Finally, if $\gamma_t(G)=2$, then Theorem \ref{general_bound(2,2,2)} leads to $3\le \gamma_{(2,2,2)}(G)\le 4$. Therefore, if $\delta=1$ or $\gamma_{\times 2,t}(G)\geq 4$, then Theorem \ref{teo-case=3-(2,2,2)} leads to  $\gamma_{(2,2,2)}(G)=4$.
\end{proof}




Theorem \ref{general_bound(2,2,2)} implies the next result. 

\begin{corollary}\label{Cycles(2,2,2)}
For any integer $n\geq 3$,
$$\gamma_{(2,2,2)}(C_n)=n.$$
\end{corollary}

In order to give the value of $\gamma_{(2,2,2)}(P_n)$, we recall the following well-known result. 

\begin{proposition} {\rm \cite{Henning2013} }\label{TotalDominationCyclesPaths}
For any integer $n\ge 3$,
$$\gamma_t(P_n)=
\left\{ \begin{array}{ll}
\frac{n}{2} & if \, n\equiv 0\pmod 4,\\[5pt]
\frac{n+1}{2} & if \,  n\equiv 1,3\pmod 4,\\[5pt]
\frac{n}{2}+1  & if \,  n\equiv 2\pmod 4.
\end{array}\right. $$
\end{proposition}

\begin{lemma}\label{lem-new-new}
If $P_n=u_1u_2\ldots u_n$ is a path  of order $n\geq 6$, then  there exists a $\gamma_{(2,2,2)}(P_n)$-function $f$ such that $f(u_n)=f(u_{n-3})=0$ and $f(u_{n-1})=f(u_{n-2})=2$. 
\end{lemma}

\begin{proof}
Let $f(V_0,V_1,V_2)$ be a $\gamma_{(2,2,2)}(P_n)$-function such that $|V_2|$ is maximum. Since $u_n$ is a leaf,  $f(u_{n-1})=2$ . Notice that $f(u_n)+f(u_{n-2})\geq 2$. Hence,  we can assume that $f(u_{n-2})=2$ and $f(u_n)=0$. Now, if $f(u_{n-3})>0$, then we can define a $(2,2,2)$-dominating function  $f'$ by $f'(u_{n-3})=0$, $f'(u_{n-5})=\min\{2,f(u_{n-5})+f(u_{n-3})\}$ and $f'(u_i)=f(u_i)$ for the remaining cases. Since $\omega(f')\le \omega(f)=\gamma_{(2,2,2)}(P_n)$, either $f'$ is a $\gamma_{(2,2,2)}(P_n)$-function with $f'(u_{n-3})=0$ or $f(u_{n-3})=0$. In both cases the result follows. 
\end{proof}

\begin{proposition}\label{teo-Pn}
For any integer $n\geq 3$,
$$\gamma_{(2,2,2)}(P_n)=2\gamma_t(P_n)=\left\{ \begin{array}{ll}
n & if \, n\equiv 0\pmod 4,\\[5pt]
n+1 & if \,  n\equiv 1,3\pmod 4,\\[5pt]
n+2  & if \,  n\equiv 2\pmod 4.
\end{array}\right. $$
\end{proposition}

\begin{proof}
Since Theorem \ref{general_bound(2,2,2)} leads to $\gamma_{(2,2,2)}(P_n)\leq 2\gamma_t(P_n)$, we only need to prove that $\gamma_{(2,2,2)}(P_n)\geq 2\gamma_t(P_n)$. We proceed by induction on $n$. It is easy to check that $\gamma_{(2,2,2)}(P_n)=2\gamma_t(P_n)$ for $n=3,4,5,6$. This establishes the base case. Now, we assume that $n\geq 7$ and  $\gamma_{(2,2,2)}(P_k)\geq 2\gamma_t(P_k)$ for $k<n$. Let $f(V_0,V_1,V_2)$ be a $\gamma_{(2,2,2)}(P_n)$-function  which satisfies Lemma \ref{lem-new-new},  and let $f'$ be the restriction of $f$ to $V(P_{n-4})$, where 
$P_{n}=u_1u_2\ldots u_{n}$ and $P_{n-4}=u_1u_2\ldots u_{n-4}$. Hence,  by applying the induction hypothesis,
$$\gamma_{(2,2,2)}(P_n)=\omega(f)=\omega(f')+4\geq \gamma_{(2,2,2)}(P_{n-4})+4\geq 2\gamma_t(P_{n-4})+4\geq 2\gamma_t(P_n).$$
 To conclude the proof we apply Proposition \ref{TotalDominationCyclesPaths}.
\end{proof}

\subsection{Preliminary results on $(2,2,1)$-domination}
\label{SubSection(2,2,1)}

\begin{theorem}\label{Bounds(2,2,1)}
For any graph $G$ with no isolated vertex, order $n$ and maximum degree $\Delta$, 
  $$  \left\lceil  \frac{2n+\gamma_t(G)}{\Delta+1} \right\rceil\le \gamma_{(2,2,1)}(G) \le  \min\{3\gamma(G), 2\gamma_t(G)\}.$$
  Furthermore, if $G$ has minimum degree $\delta\geq 2$, then
$$\gamma_{(2,2,1)}(G)\leq \gamma_{\times 2,t}(G).$$
\end{theorem}

\begin{proof} 
In order to prove the upper bound $\gamma_{(2,2,1)}(G)\le 2\gamma_t(G)$, we apply Remark \ref{REmarkMonotonicity} and Theorem \ref{general_bound(2,2,2)}, i.e.,   $\gamma_{(2,2,1)}(G)\le \gamma_{(2,2,2)}\le 2\gamma_t(G)$. 


Now, let $S$ be a $\gamma(G)$-set. Since $G$ does not have isolated vertex, for each vertex $x\in S$ such that $N(x)\cap S=\varnothing$, we fix one vertex $x'\in N(x)$ and we form a set $S'$ from them. Hence, $S\cup S'$ is a total dominating set and $|S\cup S'|=|S|+|S'|\le 2\gamma(G)$. 
Notice that the function $g(X_0,X_1,X_2)$ defined by $X_2=S$ and $X_1=S'$, is a $(2,2,1)$-dominating function on $G$. Thus, $\gamma_{(2,2,1)}(G) \le \omega(g)=2|S|+|S'|\le 3\gamma(G)$, and so  $\gamma_{(2,2,1)}(G) \le \min\{2\gamma_t(G), 3\gamma(G)\}$.

On the other side, if $G$ has minimum degree $\delta\geq 2$, then by
Corollary \ref{CorollaryGeneralUpperBounds} (ii) we have that  $\gamma_{(2,2,1)}(G)\le\gamma_{\times 2,t}(G).$

In order to prove the lower bound, let $f(V_0,V_1,V_2)$ be a $\gamma_{(2,2,1)}(G)$-function.
Since $V_1\cup V_2$ is a total dominating set,  $\gamma_t(G)\le |V_1|+|V_2|$.  Furthermore, from Lemma \ref{LemmaCotasINf} we have, $2n-|V_2|\le \Delta\gamma_{(2,2,1)}(G)$, which implies that
$2n+\gamma_t(G)\le 2n+|V_1|+|V_2|\le \Delta\gamma_{(2,2,1)}(G)+|V_1|+2|V_2|= (\Delta+1)\gamma_{(2,2,1)}(G).$
Therefore, the lower bound follows.
\end{proof}

The bounds above are tight. For instance, the graph shown in Figure~\ref{figDominat=2packing} satisfies  $\gamma_{(2,2,1)}(G)=3\gamma(G)=9$.  Next we show that the remaining two bounds are also achieved. 

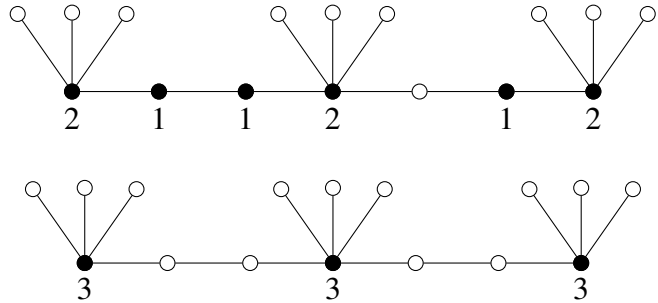
\begin{figure}[ht]
\centering
\begin{tikzpicture}[scale=2.1]
\foreach \ind in {1,...,7}
{
\pgfmathparse{.55*(\ind-1)};
\node[draw, shape=circle, fill=white, scale=.5] at (\pgfmathresult,0)(v\ind){};
}

\foreach \ind in {1,4,7}
{
\node[draw, shape=circle, fill=black, scale=.5] at (v\ind){};
\node at ([shift={(0,-.17)}]v\ind) {$2$};
}
\foreach \ind in {2,3,6}
{
\node[draw, shape=circle, fill=black, scale=.5] at (v\ind){};
\node at ([shift={(0,-.17)}]v\ind) {$1$};
}

\node[draw, shape=circle, fill=white, scale=.5] at ([shift={(55:0.6)}]v1)(v8){};
\node[draw, shape=circle, fill=white, scale=.5] at ([shift={(90:0.5)}]v1)(v9){};
\node[draw, shape=circle, fill=white, scale=.5] at ([shift={(125:0.6)}]v1)(v10){};
\node[draw, shape=circle, fill=white, scale=.5] at ([shift={(55:0.6)}]v4)(v11){};
\node[draw, shape=circle, fill=white, scale=.5] at ([shift={(90:0.5)}]v4)(v12){};
\node[draw, shape=circle, fill=white, scale=.5] at ([shift={(125:0.6)}]v4)(v13){};
\node[draw, shape=circle, fill=white, scale=.5] at ([shift={(55:0.6)}]v7)(v14){};
\node[draw, shape=circle, fill=white, scale=.5] at ([shift={(90:0.5)}]v7)(v15){};
\node[draw, shape=circle, fill=white, scale=.5] at ([shift={(125:0.6)}]v7)(v16){};
\foreach \ind in {1,...,6}
{
\pgfmathtruncatemacro{\index}{\ind+1};
\draw (v\ind)--(v\index);
}
\foreach \ind in {8,...,10}
{
\draw (v1)--(v\ind);
}
\foreach \ind in {11,...,13}
{
\draw (v4)--(v\ind);
}
\foreach \ind in {14,...,16}
{
\draw (v7)--(v\ind);
}

\end{tikzpicture}

\vspace{0.5cm}

\begin{tikzpicture}[scale=2]

\foreach \ind in {17,...,23}
{
\pgfmathparse{.55*(\ind-1)};
\node[draw, shape=circle, fill=white, scale=.5] at (\pgfmathresult-4.4,0)(v\ind){};
}

\foreach \ind in {17,20,23}
{
\node[draw, shape=circle, fill=black, scale=.5] at (v\ind){};
\node at ([shift={(0,-.17)}]v\ind) {$3$};
}
\node[draw, shape=circle, fill=white, scale=.5] at ([shift={(55:0.6)}]v17)(v24){};
\node[draw, shape=circle, fill=white, scale=.5] at ([shift={(90:0.5)}]v17)(v25){};
\node[draw, shape=circle, fill=white, scale=.5] at ([shift={(125:0.6)}]v17)(v26){};
\node[draw, shape=circle, fill=white, scale=.5] at ([shift={(55:0.6)}]v20)(v27){};
\node[draw, shape=circle, fill=white, scale=.5] at ([shift={(90:0.5)}]v20)(v28){};
\node[draw, shape=circle, fill=white, scale=.5] at ([shift={(125:0.6)}]v20)(v29){};
\node[draw, shape=circle, fill=white, scale=.5] at ([shift={(55:0.6)}]v23)(v30){};
\node[draw, shape=circle, fill=white, scale=.5] at ([shift={(90:0.5)}]v23)(v31){};
\node[draw, shape=circle, fill=white, scale=.5] at ([shift={(125:0.6)}]v23)(v32){};
\foreach \ind in {17,...,22}
{
\pgfmathtruncatemacro{\index}{\ind+1};
\draw (v\ind)--(v\index);
}
\foreach \ind in {24,...,26}
{
\draw (v17)--(v\ind);
}
\foreach \ind in {27,...,29}
{
\draw (v20)--(v\ind);
}
\foreach \ind in {30,...,32}
{
\draw (v23)--(v\ind);
}
\end{tikzpicture}
\caption{
This figure shows a
$\gamma_{(2,2,1)}(G)$-function and a $ \gamma_{(2,2,2,0)}(G)$-function on the same graph.}\label{figDominat=2packing} 
\end{figure}

\begin{corollary}
Let $G$ be a graph with no isolated vertex, order $n$ and maximum degree $\Delta$. If  
 $\gamma_t(G)<\frac{n+\Delta+1}{\Delta+1/2}$, then   $$\gamma_{(2,2,1)}(G)=2\gamma_t(G)\; \;\text{ or } \;\; \gamma_{(2,2,1)}(G)=\left\lceil  \frac{2n+\gamma_t(G)}{\Delta+1} \right\rceil .$$
\end{corollary}

\begin{proof}
If  $\gamma_{(2,2,1)}(G)\ne \left\lceil  \frac{2n+\gamma_t(G)}{\Delta+1} \right\rceil$ and  $ \gamma_{(2,2,1)}(G)\ne2\gamma_t(G)$, then by Theorem \ref{Bounds(2,2,1)}  we deduce that
 $  \left\lceil  \frac{2n+\gamma_t(G)}{\Delta+1} \right\rceil +1 \le \gamma_{(2,2,1)}(G) \le 2\gamma_t(G)-1$, which implies that $\gamma_t(G)\ge  \frac{n+\Delta+1}{\Delta+1/2}$.
Therefore, the result follows.
\end{proof}

For the graphs $G_2$ and $G_3$ illustrated in Figure \ref{ZZZ} we have that $\gamma_t(G_2)=2<\frac{22}{9}=\frac{n+\Delta+1}{\Delta+1/2}$ and $\gamma_t(G_3)=4<\frac{32}{7}=\frac{n+\Delta+1}{\Delta+1/2}$. Notice that, 
$\gamma_{(2,2,1)}(G_2)=3=\left\lceil  \frac{2n+\gamma_t(G_2)}{\Delta+1} \right\rceil$  and $\gamma_{(2,2,1)}(G_3)=8=2\gamma_t(G_3)$.

Below we characterize the graphs with $\gamma_{(2,2,1)}(G)=3$.

\begin{theorem}\label{teo-case=3-(2,2,1)}
For a graph $G$ with no isolated vertex, the following statements are equivalent.
\begin{enumerate}
\item[{\rm (i)}] $\gamma_{(2,2,1)}(G)=3.$
\item[{\rm (ii)}] $\gamma(G)=1$ or $\gamma_{\times 2,t}(G)=3$. 
\end{enumerate}
\end{theorem}

\begin{proof}
Assume first that $\gamma_{(2,2,1)}(G)=3$, and let $f(V_0,V_1,V_2)$ be a $\gamma_{(2,2,1)}(G)$-function. If $V_2\ne \varnothing$, then $V_2$ is a dominating set of cardinality one. Hence, $\gamma(G)=1$. Now, if $V_2= \varnothing$, then $V_1$ is a double total dominating set of cardinality three. Thus,  $\gamma_{\times 2,t}(G)=3$. 

On the other side, by Theorem \ref{Bounds(2,2,1)} we have that $3\leq \left\lceil  \frac{2n+\gamma_t(G)}{\Delta+1} \right\rceil\leq \gamma_{(2,2,1)}(G)\leq 3\gamma(G)$.
Hence, if $\gamma(G)=1$, then $\gamma_{(2,2,1)}(G)=3$. 
Now, if $\gamma_{\times 2,t}(G)=3$, then $G$ has minimum degree $\delta\ge 2$ and by Theorem~\ref{Bounds(2,2,1)} we have that $\gamma_{(2,2,1)}(G)\leq \gamma_{\times 2,t}(G)=3$. Therefore, $\gamma_{(2,2,1)}(G)=3$. 
\end{proof}

Next we consider the case of graphs with $\gamma_{(2,2,1)}(G)=4$. 

\begin{theorem}\label{teo-case=4-(2,2,1)}
For a graph $G$, the following statements are equivalent.
\begin{enumerate}
\item[{\rm (i)}]  $\gamma_{(2,2,1)}(G)=4$.  
\item[{\rm (ii)}]  $\gamma_t(G)=\gamma(G)=2$ or  $\gamma_{\times 2,t}(G)=4$.
\end{enumerate}
\end{theorem}

\begin{proof}
Assume $\gamma_{(2,2,1)}(G)=4$. Notice that $G$ does not have isolated vertices and, by Theorem \ref{Bounds(2,2,1)}, we have that $\gamma(G)\ge 2$. Let $f(V_0,V_1,V_2)$ be a $\gamma_{(2,2,1)}(G)$-function. 
If $V_2=\varnothing$, then $V_1$ is a double total dominating set of cardinality four. Hence, $3\le \gamma_{\times 2,t}(G)\le |V_1|=4$, and Theorem \ref{teo-case=3-(2,2,1)} implies that $\gamma_{\times 2,t}(G)=4$.

From now on, assume that $|V_2|\in \{1,2\}$. If $|V_2|=2$, then $V_1=\varnothing$ and, as a result, $V_2$ is a total dominating set of $G$, which implies that $\gamma_t(G)=\gamma(G)=2$. Now, if  $|V_2|=1$, then $|V_1|=2$ and both vertices belonging to $V_1$ are adjacent to the vertex of weight two, and every $v\in V_0$ satisfies  $N(v)\cap V_2\neq \varnothing$ or $V_1\subseteq N(v)$. This implies that the union of $V_2$ with a singleton subset of $V_1$  forms a total dominating set of $G$, and again $\gamma_t(G)=\gamma(G)=2$. 

Conversely, if $\gamma_{\times 2,t}(G)=4$, then $G$ has minimum degree $\delta\ge 2$ and by Theorem  \ref{Bounds(2,2,1)} we have that $3\le \gamma_{(2,2,1)}(G)\le  \gamma_{\times 2,t}(G) = 4$. Hence, by Theorem \ref{teo-case=3-(2,2,1)} we deduce that  $\gamma_{(2,2,1)}(G)=4$.
Finally, if $\gamma_t(G)=2$, then Theorem \ref{Bounds(2,2,1)} leads to $3\le \gamma_{(2,2,1)}(G)\le 4$. Therefore, if $\gamma(G)=2$  then by  Theorem  \ref{teo-case=3-(2,2,1)} we conclude that  $\gamma_{(2,2,1)}(G)=4$.
\end{proof}

\begin{lemma}\label{Lemma-Pn-(2,2,1)} For any integer $n\ge 3$,
$$\gamma_{(2,2,1)}(P_n)\le\left\{ \begin{array}{ll}
n-\left\lfloor \frac{n}{7} \right\rfloor +1 & \text{if} \,\, n\equiv 1,2 \pmod 7,
\\[5pt]
n-\left\lfloor \frac{n}{7} \right\rfloor & otherwise.
\end{array}\right. $$ 
\end{lemma}
 
 \begin{proof}
First  we show how to construct a $(2,2,1)$-dominating function $f$ on $P_n$ for $n\in \{2,\dots, 8\}.$
\begin{itemize}
\item $n=2$: $f(u_1)=2$ and $f(u_2)=1$.
\item $n=3$: $f(u_1)=0$, $f(u_2)=2$ and $f(u_3)=1$.
\item $n=4$:  $f(u_1)=f(u_4)=0$ and $f(u_2)=f(u_3)=2$.
\item $n=5$:  $f(u_1)=f(u_5)=0$, $f(u_2)=f(u_4)=2$ and $f(u_3)=1$.
\item $n=6$:  $f(u_1)=f(u_6)=0$, $f(u_2)=f(u_5)=2$ and $f(u_3)=f(u_4)=1$.
\item $n=7$:  $f(u_1)=f(u_4)=f(u_7)=0$, $f(u_2)=f(u_6)=2$ and $f(u_3)=f(u_5)=1$.
\item $n=8$: $f(u_1)=f(u_4)=f(u_8)=0$, $f(u_2)=f(u_6)=f(u_7)=2$, $f(u_3)=f(u_5)=1$.
\end{itemize}
We now proceed to describe the construction of $f$ for any $n=7q+r$, where $q\ge 1$ and $0\le r \le 6$. We partition $V(P_n)=\{u_1, \dots, u_n\}$ into $q$ sets of cardinality $7$ and for $r\ge 1$  one additional set of cardinality $r$, in such a way that the subgraph induced by all these sets are paths.
 
For any $r\ne 1$, the restriction of $f$ to each of these $q$ paths of length $7$ corresponds to the weights associated above with $P_7$, while for the path of length $r$ (if any) we take the weights associated above with $P_r$. The case $r=1$ and $q\ge 2$ is slightly different, as for the first $q-1$ paths of length $7$ we take the weights associated above with $P_7$ and for the last  $8$ vertices of $P_n$ we take the weights associated above with $P_8$.

Notice that, for $n\equiv 1,2\pmod 7$, we have that $\gamma_{(2,2,1)}(P_n)\le \omega(f)=6q+r+1=n-\left\lfloor \frac{n}{7} \right\rfloor +1$, while for $n\not \equiv 1,2\pmod 7$ we have   $\gamma_{(2,2,1)}(P_n)\le \omega(f)=6q+r=n-\left\lfloor \frac{n}{7} \right\rfloor$. Therefore, the result follows.
 \end{proof}

\begin{lemma}\label{lema-1}
Let $P_7=x_1 \dots  x_7$ be a subgraph of $C_n$ and $X=\{x_1,\ldots ,x_7\}$. If  $f$ is a $(2,2,1)$-dominating function on $C_n$, then $$f(X)\geq 6.$$
\end{lemma}

\begin{proof}
Notice that $f(\{x_1,x_2,x_3\})\geq 2$ and $f(\{x_4,x_5,x_6,x_7\})\geq 3$ as $f$ is a $(2,2,1)$-dominating function. If $f(\{x_1,x_2,x_3\})\geq 3$, then we are done. Hence, we assume that $f(\{x_1,x_2,x_3\})=2$. In this case, it is not difficult to  deduce that $f(\{x_4,x_5,x_6,x_7\})\geq 4$, which implies that $f(X)\geq 6$, as desired. Therefore, the proof is complete.  
\end{proof}

\begin{lemma}\label{Lemma-Cn-(2,2,1)}
For any integer $n\geq 3$,

$$
\gamma_{(2,2,1)}(C_n)\geq\left\{ \begin{array}{ll}
             n-\lfloor\frac{n}{7}\rfloor+1 & \text{if } \,  n\equiv 1,2 \pmod 7,\\[5pt]
             n-\lfloor\frac{n}{7}\rfloor & \mbox{otherwise.}
                                  \end{array}\right.
$$
\end{lemma}

\begin{proof}
It is easy to check that $\gamma_{(2,2,1)}(C_n)=n$ for every $n\in\{3,4,5,6\}$. Now, let $n=7q+r$, with $0\leq r\leq 6$ and $q\geq 1$. Let $f(V_0,V_1,V_2)$ be a $\gamma_{(2,2,1)}(C_n)$-function.

If $r=0$, then by Lemma \ref{lema-1} we have that $\omega(f)\geq 6q=n-\lfloor\frac{n}{7}\rfloor.$  From now on we assume that $r\ge 1$. By Proposition \ref{obs-subgraph-general} and Lemma \ref{Lemma-Pn-(2,2,1)} we deduce that $\gamma_{(2,2,1)}(C_n)\le \gamma_{(2,2,1)}(P_n)<n$, which implies  that  $V_2\neq \varnothing$, otherwise there  exists $u\in V(C_n)=V_0\cup V_1$ such that  $N(u)\cap V_0\ne \varnothing$ and so $|N(u)\cap V_1|\le 1$, which is a contradiction. Let $x\in V_2$ and, without loss of generality, we can label the vertices of $C_n$ in such a way that $x=u_{1}$, and  $u_{2}\in V_1\cup V_2$ whenever $r\ge 2$. We partition $V(C_n)$ into $X=\{u_1,\ldots  ,u_r\}$ and $Y=\{u_{r+1},\ldots ,u_n\}$. Notice that Lemma \ref{lema-1} leads to $f(Y)\geq 6q$.
 
Now, if $r\in \{1,2\}$, then $f(X)\ge r+1$, which implies that $\omega(f)\ge r+1+6q=n-\lfloor\frac{n}{7}\rfloor+1$. Analogously, if $r=3$, then $f(X)\ge r$ and so 
  $\omega(f)\ge r+6q=n-\lfloor\frac{n}{7}\rfloor$. 

Finally, if $r\in \{4,5,6\}$, then as $f$ is a $(2,2,1)$-dominating function we deduce that $f(X)\ge r$, which implies that $\omega(f)\ge r+6q=n-\lfloor\frac{n}{7}\rfloor$.
\end{proof}

The following result is a direct consequence of Proposition \ref{obs-subgraph-general} and Lemmas \ref{Lemma-Pn-(2,2,1)}  and \ref{Lemma-Cn-(2,2,1)}.

\begin{proposition}\label{teo-Pn-Cn-(2,2,1)} For any integer $n\ge 3$,
$$\gamma_{(2,2,1)}(C_n)=\gamma_{(2,2,1)}(P_n)= \left\{ \begin{array}{ll}
             n-\lfloor\frac{n}{7}\rfloor+1 & \text{if } \,  n\equiv 1,2 \pmod 7,\\[5pt]
             n-\lfloor\frac{n}{7}\rfloor & \mbox{otherwise.}
                                  \end{array}\right.$$
\end{proposition}

\subsection{Preliminary results on $(2,2,0)$-domination}
\label{SubSection(2,2,0)}

\begin{theorem}\label{Bounds(2,2,0)}
For any graph $G$ with no isolated  vertex,  order $n$ and maximum degree $\Delta$,
  $$  \left\lceil  \frac{2 n}{\Delta+1} \right\rceil\le \gamma_{(2,2,0)}(G) \le 2\gamma(G).$$
  Furthermore, if $G$ has minimum degree $\delta\geq 2$, then $$\gamma_{(2,2,0)}(G)\le\gamma_{\times 2,t}(G).$$
\end{theorem}

\begin{proof} 
The upper bound $\gamma_{(2,2,0)}(G) \le \omega(g)= 2\gamma(G)$ is derived by we applying 
Theorem~\ref{GeneralUpperBounds(2,2,2)} (ii) for $i=1$ and $l=2$.  Furthermore, if $G$ has minimum degree $\delta\geq 2$, then by
Corollary \ref{CorollaryGeneralUpperBounds} (ii) we have that  $\gamma_{(2,2,0)}(G)\le\gamma_{\times 2,t}(G).$

Now, let $f(V_0,V_1,V_2)$ be a $\gamma_{(2,2,0)}(G)$-function.
From Lemma \ref{LemmaCotasINf} we deduce  that $  
2(n-|V_2|))\le \Delta\gamma_{(2,2,0)}(G),$ which implies that $2n\le 
2n+|V_1|\le (\Delta+1)\gamma_{(2,2,0)}(G).$
Therefore, the result follows.
\end{proof}
 
Theorem \ref{Bounds(2,2,0)} implies that, 
if $\gamma(G)=\frac{n}{\Delta+1}$, then $\gamma_{(2,2,0)}(G)=\frac{2n}{\Delta+1}.$ It is easy to see that a graph satisfies $\gamma(G)=\frac{n}{\Delta+1}$ if and only if there exists a $\gamma(G)$-set $S$ which is a $2$-packing\footnote{A set $S\subseteq V(G)$ is a $2$-packing if $N[u]\cap N[v]=\varnothing$ for every pair of different vertices $u,v\in S$.} and every vertex in $S$ has degree $\Delta$.
The upper bound $\gamma_{(2,2,0)}(G) \le 2\gamma(G)$ is achieved for the graph
$G$ shown in Figure \ref{FigRaro}, which satisfies  $\gamma_{(2,2,0)}(G)=2\gamma(G)=6$. Furthermore, by Theorem \ref{IgualdadCorona} we have that for any corona graph $G\cong G_1\odot G_2$, where $G_1$ does not have isolated vertices, $\gamma_{(2,2,0)}(G)=2\gamma(G)$.

As shown in Theorem \ref{teo-case=2}, for a graph $G$,  $\gamma_{(2,2,0)}(G)=2$  if and only if $\gamma(G)=1$.
Now we consider the case $\gamma_{(2,2,0)}(G)=3$.

\begin{theorem}\label{teo-case=3-(2,2,0)}
For a graph $G$, $\gamma_{(2,2,0)}(G)=3$  if and only if $\gamma_{\times 2,t}(G)=\gamma(G)+1=3$. 
\end{theorem}

\begin{proof}
Assume $\gamma_{(2,2,0)}(G)=3$. By Theorem \ref{teo-case=2} we have that $\gamma(G)\ge 2$. Let $f(V_0,V_1,V_2)$ be a $\gamma_{(2,2,0)}(G)$-function. If $|V_2|=1$ then $|V_1|=1$, and as $f$ is a $(2,2,0)$-dominating function we deduce that $N[V_2]=V(G)$, i.e., $\gamma(G)=1$, which is a contradiction. Thus, $V_2=\varnothing$ and $|V_1|=3$. 
Notice that $V_1$ is a double total dominating set and since $\gamma(G)\ge 2$, it follows that $3\leq\gamma(G)+1\leq \gamma_{\times 2,t}(G)\le |V_1|=3$. Hence, $\gamma_{\times 2,t}(G)=\gamma(G)+1=3$, as required.

Conversely, assume $\gamma_{\times 2,t}(G)=\gamma(G)+1=3$. Since $G$ has minimum degree at least two, Theorem~\ref{Bounds(2,2,0)}
leads to $2\leq \gamma_{(2,2,0)}(G)\leq \gamma_{\times 2,t}(G)=3$, and so Theorem  \ref{teo-case=2} implies that $\gamma_{(2,2,0)}(G)=3$, which completes the proof.
\end{proof}

\begin{theorem}\label{teo-case=4-(2,2,0)}
For a graph $G$, $\gamma_{(2,2,0)}(G)=4$  if and only if
one of the following conditions holds.
\begin{enumerate}
\item[{\rm (i)}] $G\cong K_1\cup G_1$, where $G_1$ is a graph with $\gamma(G_1)=1$.
\item[{\rm (ii)}] $\gamma_{\times 2,t}(G)=4$.
\item[{\rm (iii)}] $\gamma(G)=2$ and $G$ has minimum degree one.
\item[{\rm (iv)}] $\gamma(G)=2$ and $\gamma_{\times 2,t}(G)\geq 4$.
\end{enumerate}
\end{theorem}

\begin{proof}If $K_1$ is a component of $G$, then by  Theorem \ref{teo-case=2} we conclude that 
$\gamma_{(2,2,0)}(G)=4$  if and only if
 $G\cong K_1\cup G_1$, where $G_1$ is a graph with $\gamma(G_1)=1$.

 From now on, we consider the case where $G$ is a graph  with no isolated vertex.
Assume $\gamma_{(2,2,0)}(G)=4$ and let $f(V_0,V_1,V_2)$ be a $\gamma_{(2,2,0)}(G)$-function. If $V_2=\varnothing$, then $V_1$  is a double total dominating set of $G$. In this case, $G$ has minimum degree $\delta\ge 2$ and by Theorem~\ref{Bounds(2,2,0)} we have  that $\gamma_{\times 2,t}(G)\leq |V_1|=4=\gamma_{(2,2,0)}(G)\leq \gamma_{\times 2,t}(G)$. Hence  (ii) follows. 

Now, assume that $|V_2|\in \{1,2\}$. 
If $|V_2|=2$, then $V_1=\varnothing$, and so $\gamma(G)\le 2$. 
Now, if  $|V_2|=1$, then $|V_1|=2$ and both vertices belonging $V_1$ are adjacent to the vertex of weight two, and every $v\in V_0$ satisfies  $N(v)\cap V_2\neq \varnothing$ or $V_1\subseteq N(v)$. This implies that the union of $V_2$ with a singleton subset of $V_1$  forms a dominating set of $G$, and again $\gamma(G)\le 2$. 
Thus, from Theorem \ref{teo-case=2} we deduce that $\gamma(G)=2$.  Furthermore, if  $\delta\ge 2$, then by  Theorem~\ref{Bounds(2,2,0)} we have   that $\gamma_{\times 2,t}(G)\geq  \gamma_{(2,2,0)}=4$. Therefore, either (iii) or  (iv) holds.

Conversely, if $\gamma_{\times 2,t}(G)=4$, then Theorem~\ref{Bounds(2,2,0)} leads to $2 \le \gamma_{(2,2,0)}\le  \gamma_{\times 2,t}(G)=4$. Hence, by Theorems \ref{teo-case=2} and \ref{teo-case=3-(2,2,0)} we deduce that  $\gamma_{(2,2,0)}(G)=4$. Analogously, if $\gamma(G)=2$ and $\delta\ge 1$, then
  Theorem~\ref{Bounds(2,2,0)} leads to $2 \le \gamma_{(2,2,0)}\le  2\gamma(G)=4$. Thus, by Theorem  \ref{teo-case=2} we have that $3 \le \gamma_{(2,2,0)}\le 4$. In particular, if $\delta=1$ or $\gamma_{\times 2,t}(G)\ge 4$, then 
  Theorem \ref{teo-case=3-(2,2,0)} leads to 
   $\gamma_{(2,2,0)}(G)=4$,  which completes the proof.
\end{proof}

 \begin{lemma}\label{lem-(2,2,0)}
For a graph $G$, the following statements are equivalent.
\begin{enumerate}
\item[{\rm (i)}] $\gamma_{(2,2,0)}(G)=2\gamma(G)$.
\item[{\rm (ii)}] There exists a $\gamma_{(2,2,0)}(G)$-function $f(V_0,V_1,V_2)$ such that $V_1=\varnothing$.
\end{enumerate}
\end{lemma}

\begin{proof}
First, we assume that $\gamma_{(2,2,0)}(G)=2\gamma(G)$ and let $D$ be a $\gamma(G)$-set. Hence, the function $f(V_0,V_1,V_2)$, defined by $V_2=D$ and $V_0=V(G)\setminus D$, is a $\gamma_{(2,2,0)}(G)$-function which satisfies (ii), as desired.

Finally, we assume that there exists a $\gamma_{(2,2,0)}(G)$-function $f(V_0,V_1,V_2)$ such that $V_1=\varnothing$. This implies that $V_2$ is a dominating set of $G$. Hence, $\gamma_{(2,2,0)}(G)\leq 2\gamma(G)\leq 2|V_2|=\gamma_{(2,2,0)}(G)$, and the desired equality holds, which completes the proof.    
\end{proof}

  The following result provides the $(2,2,0)$-domination number of paths and cycles.  
 
 \begin{proposition}\label{teo-Cn-Pn-(2,2,0)}
For any integer $n\geq 3$,
$$\gamma_{(2,2,0)}(P_n)=\gamma_{(2,2,0)}(C_n)=2\left\lceil \frac{n}{3}\right\rceil.$$
\end{proposition}

\begin{proof}
We first prove that $\gamma_{(2,2,0)}(C_n)\geq 2\left\lceil \frac{n}{3}\right\rceil.$ Let $f(V_0,V_1,V_2)$ be a $\gamma_{(2,2,0)}(C_n)$-function. If $V_1=\varnothing$, then by Lemma \ref{lem-(2,2,0)} it follows that $\gamma_{(2,2,0)}(C_n)=2\gamma(C_n)=2\left\lceil \frac{n}{3}\right\rceil.$ If $V_1\neq \varnothing$, then $1+2|V_2|\leq |V_1|+2|V_2|=\gamma_{(2,2,0)}(C_n)\leq 2\gamma(C_n)=2\left\lceil \frac{n}{3}\right\rceil$, which leads to $|V_2|\leq \left\lceil \frac{n}{3}\right\rceil-1$. By Lemma \ref{LemmaCotasINf} we have that $\gamma_{(2,2,0)}(C_n)\geq n-|V_2|\geq n-\left\lceil \frac{n}{3}\right\rceil+1\geq 2\left\lceil \frac{n}{3}\right\rceil $, as desired.

Therefore, by the inequality above, Proposition \ref{obs-subgraph-general} and Theorem \ref{Bounds(2,2,0)} we deduce that $2\lceil\frac{n}{3}\rceil\leq \gamma_{(2,2,0)}(C_n)\leq \gamma_{(2,2,0)}(P_n)\leq 2\gamma(P_n)=2\lceil\frac{n}{3}\rceil$. Thus, we have equalities in the inequality chain above, which implies that the result follows. 
\end{proof}

\subsection{Preliminary results on $(2,1,0)$-domination}\label{SubSection(2,1,0)}

Given a graph $G$,  we use the notation $L(G)$ and $S(G)$ for the sets of leaves and support vertices, respectively. 

\begin{theorem}\label{Bounds(2,1,0)}
For any graph $G$ with no isolated  vertex,  order $n$ and maximum degree $\Delta$, 
  $$  \left\lceil  \frac{2n}{\Delta+1} \right\rceil\le \gamma_{(2,1,0)}(G) \le \min\{\gamma_{\times 2}(G)-|L(G)|+|S(G)|,2\gamma(G)\}.$$
\end{theorem}

\begin{proof} 
If $f(V_0,V_1,V_2)$ is a $\gamma_{(2,1,0)}(G)$-function, then from  Lemma \ref{LemmaCotasINf} we conclude that $2n-|V_1|-2|V_2|\le  
 \Delta\gamma_{(2,1,0)}(G).$ Hence, 
$2n\le \Delta\gamma_{(2,1,0)}(G)+\omega(f)=(\Delta+1)\gamma_{(2,1,0)}(G).$
Therefore, the lower bound  follows.

Let $D$ be a $\gamma_{\times 2}(G)$-set. Notice that $S(G)\cup L(G)\subseteq D$. Since $|N[v]\cap D |\ge 2$ for every $v\in V(G)$, the function $g(V_0,V_1,V_2)$ defined by $V_1=D\setminus (L(G)\cup S(G))$ and $V_2=S(G)$, is a $(2,1,0)$-dominating function. Hence, $\gamma_{(2,1,0)}(G) \le \omega(g)= \gamma_{\times 2}(G)-|L(G)|+|S(G)|$.

By Remark \ref{REmarkMonotonicity},  $\gamma_{(2,1,0)}(G)\le \gamma_{(2,2,0)}(G)$, hence the upper bound $\gamma_{(2,1,0)}(G) \le 2\gamma(G)$  is derived from  
 Theorem  \ref{Bounds(2,2,0)}.
Therefore, $\gamma_{(2,1,0)}(G)\leq \min\{\gamma_{\times 2}(G)-|L(G)|+|S(G)|,2\gamma(G)\}$.
\end{proof}
 
 The bounds above are tight. For instance, for the  graph $G_1$ shown in Figure \ref{ZZZ} we have that $ \gamma_{(2,1,0)}(G_1) = \left\lceil  \frac{2n}{\Delta+1} \right\rceil=\gamma_{\times 2}(G_1)=2\gamma(G_1)=4.$  As an example of graph of minimum degree one where $\gamma_{(2,1,0)}(G) =\gamma_{\times 2}(G)-|L(G)|+|S(G)|$ we take the graph  $G$ obtained from a  star graph $K_{1,r}$, $r\ge 3$, by subdividing one edge just once. In such a case, $\gamma_{(2,1,0)}(G)=4 =\gamma_{\times 2}(G)-|L(G)|+|S(G)|$. Another example is the graph shown in Figure \ref{FigRaro} which satisfies $\gamma_{(2,1,0)}(G) =\gamma_{\times 2}(G)-|L(G)|+|S(G)|=6$.
 
 Notice that $\gamma_{(2,1,0)}(G)\ge \left\lceil  \frac{2n}{\Delta+1} \right\rceil\ge 2$. As shown in Theorem \ref{teo-case=2},  $\gamma_{(2,1,0)}(G)=2$  if and only if $\gamma(G)=1$.
 Next we characterize the graph satisfying $\gamma_{(2,1,0)}(G)=3$.

\begin{theorem}\label{teo-case=3-(2,1,0)}
For a graph $G$, $\gamma_{(2,1,0)}(G)=3$  if and only if $\gamma_{\times 2}(G)=\gamma(G)+1=3$. 
\end{theorem}

\begin{proof}
Assume $\gamma_{(2,1,0)}(G)=3$. By Theorem \ref{teo-case=2} we have that $\gamma(G)\ge 2$. Let $f(V_0,V_1,V_2)$ be a $\gamma_{(2,1,0)}(G)$-function. If $|V_2|=1$ then  $N[V_2]=V(G)$, i.e., $\gamma(G)=1$, which is a contradiction. Thus, $V_2=\varnothing$ and $|V_1|=3$, which implies that  $V_1$ is a double dominating set. Hence, $3\leq\gamma(G)+1\leq \gamma_{\times 2}(G)\le |V_1|=3$. Therefore, $\gamma_{\times 2}(G)=\gamma(G)+1=3$.

Conversely, assume  $\gamma_{\times 2}(G)=\gamma(G)+1=3$. Notice that $G$ has minimum degree $\delta\ge 1$ and so by Theorems \ref{teo-case=2} and \ref{Bounds(2,1,0)} we have that $3\leq \gamma_{(2,1,0)}(G)\leq \gamma_{\times 2}(G)=3$, which implies that $\gamma_{(2,1,0)}(G)=3$.
\end{proof}

Next we consider the case of graphs with $\gamma_{(2,1,0)}(G)=4$. 

\begin{theorem}\label{teo-case=4-(2,1,0)}
For a graph $G$, $\gamma_{(2,1,0)}(G)=4$  if and only if
one of the following conditions is satisfied.
\begin{enumerate}
\item[{\rm (i)}] $G\cong K_1\cup G_1$, where $G_1$ is a graph with $\gamma(G_1)=1$.
\item[{\rm (ii)}] $\gamma_{\times 2}(G)=4$.
\item[{\rm (iii)}] $\gamma(G)=2$ and $\gamma_{\times 2}(G)\geq 4$.
\end{enumerate}
\end{theorem}

\begin{proof}
If $K_1$ is a component of $G$, then by  Theorem \ref{teo-case=2} we conclude that 
$\gamma_{(2,1,0)}(G)=4$  if and only if
 $G\cong K_1\cup G_1$, where $G_1$ is a graph with $\gamma(G_1)=1$.

 From now on, we consider the case where $G$ is a graph  with no isolated vertex.
Assume $\gamma_{(2,1,0)}(G)=4$. By Theorem \ref{Bounds(2,1,0)} we deduce that $\gamma_{\times 2}(G)\geq 4$ and $\gamma(G)\ge 2$. Let $f(V_0,V_1,V_2)$ be a $\gamma_{(2,1,0)}(G)$-function. If $V_2=\varnothing$, then $V_1$  is a double dominating set of $G$, which implies that $\gamma_{\times 2}(G)\leq |V_1|=4$. Hence, (ii) follows. From now on, assume $|V_2|\in \{1,2\}$. If $|V_2|=2$, then $V_1=\varnothing$ and so, $V_2$ is a dominating set of $G$, which implies that $\gamma(G)=2$.
If  $|V_2|=1$, then for every $v\in V_1$ we have that $V_2\cup\{v\}$ is a dominating set of $G$. Hence, $\gamma(G)=2$.  Therefore, (iii) follows.

Conversely, if (ii) or (iii)  holds, then by Theorems \ref{Bounds(2,1,0)}  we have that 
$2\le \gamma_{(2,1,0)}(G)\le 4$. Therefore, by Theorems 
\ref{teo-case=2} and \ref{teo-case=3-(2,1,0)} we deduce that  $\gamma_{(2,1,0)}(G)=4$,  which completes the proof.
\end{proof}

The  formulas on the $\{k\}$-dominating number of cycles and paths were obtained in  \cite{k-domination-2008}. We present here the particular case of $k=2$, as $  \gamma_{\{2\}}(G)= \gamma_{(2,1,0)}(G)$.

\begin{proposition}{\rm \cite{k-domination-2008}}\label{teo-Cn-(2,1,0)}
For any integer $n\ge 3$,  $$ \gamma_{\{2\}}(C_n) =\left\lceil  \frac{2n}{3} \right\rceil \quad \text{ and } \quad \gamma_{\{2\}}(P_n)=
2\left\lceil  \frac{n}{3} \right\rceil. $$
\end{proposition}

\subsection{Preliminary results on $(2,2,2,0)$-domination}
\label{SubSection(2,2,2,0)}

The following result is a direct consequence of Theorem \ref{GeneralUpperBounds(2,2,2)} (i), (ii) and (vi).

\begin{corollary}\label{Bounds(2,2,2,0)}
For any graph $G$ with no isolated vertex, 
  $$\gamma_{(2,2,1)}(G)\le \gamma_{(2,2,2,0)}(G) \le  \min\{3\gamma(G), \gamma_{(2,2,2)}(G)\}.$$
\end{corollary}

The bounds above are tight. For instance, every graph $G_{k,r}$ belonging to the infinite family $\mathcal{H}_k$ constructed after Remark \ref{REmarkMonotonicityPart} satisfies the equalities  $\gamma_{(2,2,1)}(G_{k,r})=\gamma_{(2,2,2)}(G_{k,r})=\gamma_{(2,2,2,0)}(G_{k,r})=k$. In contrast, 
the graph  shown in Figure \ref{FigRaro} satisfies  $\gamma_{(2,2,1)}(G)=6<7=\gamma_{(2,2,2,0)}(G)<8=\gamma_{(2,2,2)}(G)$. Moreover, Figure \ref{figDominat=2packing} illustrates a graph $G$ with $\gamma_{(2,2,1)}(G)=\gamma_{(2,2,2,0)}(G)=3\gamma(G)=9.$

In order to characterize the graphs with 
$\gamma_{(2,2,2,0)}(G)\in \{3,4\}$, we need to establish the following lemma. 

\begin{lemma}\label{lem-small-values-(2,2,2,0)}
For a graph $G$, the following statements are equivalent.

\begin{enumerate}
\item[{\rm (i)}] $\gamma_{(2,2,2,0)}(G)=\gamma_{(2,2,2)}(G)$.
\item[{\rm (ii)}] There exists a $\gamma_{(2,2,2,0)}(G)$-function $f(V_0,V_1,V_2,V_3)$ such that $V_3=\varnothing  $.
\end{enumerate}
\end{lemma}

\begin{proof}  
If $\gamma_{(2,2,2,0)}(G)=\gamma_{(2,2,2)}(G)$, then for any $\gamma_{(2,2,2)}(G)$-function  $f(V_0,V_1,V_2)$, there exists a $\gamma_{(2,2,2,0)}(G)$-function $g(W_0,W_1,W_2,W_3)$ defined by  $W_0=V_0$, $W_1=V_1$, $W_2=V_2$ and $W_3=\varnothing  $.  Therefore, (i) implies (ii).

Conversely, if there exists a $\gamma_{(2,2,2,0)}(G)$-function $f(V_0,V_1,V_2,V_3)$ such that $V_3=\varnothing  $, then the function $g(W_0,W_1,W_2)$, defined by $W_0=V_0$, $W_1=V_1$ and $W_2=V_2$, is a $(2,2,2)$-dominating function on $G$, and so $\gamma_{(2,2,2)}(G)\leq \omega(g)=\omega(f)=\gamma_{(2,2,2,0)}(G)$. Therefore, Corollary \ref{Bounds(2,2,2,0)}   leads to  $\gamma_{(2,2,2,0)}(G)=\gamma_{(2,2,2)}(G)$, which completes the proof.
\end{proof}

\begin{theorem}\label{teo-case=3-(2,2,2,0)}
For a graph $G$,  the following statements are equivalent.
\begin{enumerate}
\item[{\rm (i)}] $\gamma_{(2,2,2,0)}(G)=3$.
\item[{\rm (ii)}] $\gamma(G)=1$ or $\gamma_{\times 2,t}(G)=3$.
\end{enumerate}
\end{theorem}

\begin{proof}
 Assume first that $\gamma_{(2,2,2,0)}(G)=3$, and let $f(V_0,V_1,V_2,V_3)$ be a $\gamma_{(2,2,2,0)}(G)$-function. Notice that $|V_3|\in \{0,1\}$. If $|V_3|=1$, then $V_1\cup V_2=\varnothing  $, which implies that $V_3$ is a dominating set of cardinality one. Hence, $\gamma(G)=1$.
 
If $V_3=\varnothing  $, then by Lemma \ref{lem-small-values-(2,2,2,0)} we have that $\gamma_{(2,2,2)}(G)=\gamma_{(2,2,2,0)}(G)=3$, and by Theorem \ref{teo-case=3-(2,2,2)} we deduce that $\gamma_{\times 2,t}(G)=3$. 

Conversely, if $\gamma(G)=1$, then Corollary \ref{Bounds(2,2,2,0)} leads to $3\leq \gamma_{(2,2,2,0)}(G)\leq 3\gamma(G)=3$. Moreover, if $\gamma_{\times 2,t}(G)=3$, then $G$ has minimum degree $\delta\ge 2$ and so Theorem  \ref{GeneralUpperBounds(2,2,2)} (i)  leads to $3\leq \gamma_{(2,2,2,0)}(G)\leq \gamma_{(2,2,2)}(G)\leq \gamma_{\times 2,t}(G)=3$. Therefore, $\gamma_{(2,2,2,0)}(G)=3$. 
\end{proof}

\begin{theorem}\label{teo-case=4-(2,2,2,0)}
For a graph $G$,  $\gamma_{(2,2,2,0)}(G)=4$  if and only if at least 
one of the following conditions holds.
\begin{enumerate}
\item[{\rm (i)}] $\gamma_{\times 2,t}(G)=4$.
\item[{\rm (ii)}] $\gamma(G)=\gamma_t(G)=2$ and $G$ has minimum degree $\delta=1$. 
\item[{\rm (iii)}] $\gamma(G)=\gamma_t(G)=2$ and  $\gamma_{\times 2,t}(G)\geq 4$.
\end{enumerate}
\end{theorem}

\begin{proof}
Assume $\gamma_{(2,2,2,0)}(G)=4$. Let $f(V_0,V_1,V_2,V_3)$ be a $\gamma_{(2,2,2,0)}(G)$-function. Hence, $|V_3|\in \{0,1\}$. If $ |V_3|=1$, then $V_3$ is a dominating set of cardinality one. Hence, $\gamma(G)=1$, which is a contradiction with Theorem \ref{teo-case=3-(2,2,2,0)}. Hence, $V_3=\varnothing  $, and so, Lemma \ref{lem-small-values-(2,2,2,0)} leads to $\gamma_{(2,2,2)}(G)=\gamma_{(2,2,2,0)}(G)=4$. Thus, by Theorems \ref{teo-case=4-(2,2,2)} and \ref{teo-case=3-(2,2,2,0)} we deduce (i)-(iii). 

Conversely, if conditions (i)-(iii) hold, then by Theorem \ref{teo-case=3-(2,2,2)} we have that $\gamma_{(2,2,2)}(G)=4$. Corollary~\ref{Bounds(2,2,2,0)} leads to $3\leq \gamma_{(2,2,2,0)}(G)\leq \gamma_{(2,2,2)}(G)=4$. Notice that if  $\delta\geq 2$, then $\gamma(G)\geq 2$ and $\gamma_{\times 2,t}(G)\geq 4$. Hence, Theorem \ref{teo-case=3-(2,2,2,0)} leads to  $\gamma_{(2,2,2,0)}(G)=4$.
\end{proof}

\begin{proposition}\label{prop-Cn-2220}
For any integer $n\geq 3$,
$$\gamma_{(2,2,2,0)}(C_n)=n.$$
\end{proposition}

\begin{proof}
By Corollaries \ref{Cycles(2,2,2)} and  \ref{Bounds(2,2,2,0)}  we have that $\gamma_{(2,2,2,0)}(C_n)\leq \gamma_{(2,2,2)}(C_n)= n$. We only need to prove that $\gamma_{(2,2,2,0)}(C_n)\geq n$. Let $f(V_0,V_1,V_2,V_3)$ be a $\gamma_{(2,2,2,0)}(G)$-function such that $|V_3|$ is minimum. If $V_3=\varnothing $, then by Lemma \ref{lem-small-values-(2,2,2,0)} and Corollary \ref{Cycles(2,2,2)} we conclude that $\gamma_{(2,2,2,0)}(C_n)= n$. 
Assume $V_3\ne \varnothing $.
If $v\in V_3$, then $N(v)\subseteq V_0$ as otherwise, by choosing  one vertex $u\in N(v)\setminus V_0$, the function $f'$ defined by $f'(v)=2$, $f'(u)=\min\{2,f(u)+1\}$ and $f'(x)=f(x)$ for the remaining vertices, is a $(2,2,2,0)$-dominating function with $\omega(f')\le \omega(f)$ and $|V_3'|<|V_3|$, which is a contradiction.  Hence, $\sum_{x\in V_3}f(N[x])=3|V_3|$. Now, we observe that
$$2\sum_{x\in V(C_n)\setminus N[V_3]}f(x)\ge\displaystyle\sum_{x\in V(C_n)\setminus N[V_3]}\left( \sum_{u\in N(x)}f(u)\right)\geq 2(n-3|V_3|).$$

Therefore,
$$\gamma_{(2,2,2,0)}(C_n)=\omega(f)=\sum_{x\in V_3}f(N[x])+ \sum_{x\in V(C_n)\setminus N[V_3]}f(x)\geq 3|V_3|+(n-3|V_3|)=n,$$ and the result follows.
\end{proof}

\begin{proposition}
For any integer $n\geq 3$,
$$\gamma_{(2,2,2,0)}(P_n)=\left\{ \begin{array}{ll}
6 & if \, n=5,\\[5pt]
n &  \text{ otherwise}.
\end{array}\right. $$
\end{proposition}

\begin{proof}
It is easy to check that  $\gamma_{(2,2,2,0)}(P_n)=n$ for  $n=3,4,6,7,8$, and also $\gamma_{(2,2,2,0)}(P_5)=6$. From now on,  assume $n\ge 9$. By Propositions \ref{obs-subgraph-general} and \ref{prop-Cn-2220} we have that $n=\gamma_{(2,2,2,0)}(C_n)\leq \gamma_{(2,2,2,0)}(P_n)$. Hence, we only need to prove that $\gamma_{(2,2,2,0)}(P_n)\leq n$. To this end,  we proceed to construct a $(2,2,2,0)$-dominating function $f(V_0,V_1,V_2,V_3)$ on $P_n=v_1v_2\ldots v_n$ such that 
 $\omega(f)=n$. 
\begin{itemize}
\item If $n\equiv 0\pmod 3$, then we set $V_3=\displaystyle\bigcup_{i=1}^{n/3}\{v_{3i-1}\}$ and $V_0=V(G)\setminus V_3$. 
\item If $n\equiv 1\pmod 3$, then we set $V_3=\displaystyle\bigcup_{i=1}^{(n-4)/3}\{v_{3i-1}\}$, $V_2=\{v_{n-2},v_{n-1}\}$ and $V_0=V(G)\setminus (V_2\cup V_3)$. 
\item If $n\equiv 2\pmod 3$, then we set $V_3=\displaystyle\bigcup_{i=1}^{(n-8)/3}\{v_{3i-1}\}$, $V_2=\{v_{n-6},v_{n-5}, v_{n-2},v_{n-1}\}$ and $V_1=\varnothing$. 
\end{itemize}
Notice that in the three cases above, $f$ is a $(2,2,2,0)$-dominating function of weight $\omega(f)=n$, as required. Therefore, the proof is complete.
\end{proof}

\section*{Acknowledgements}
The authors would thank the anonymous reviewers for their generous time in providing detailed comments and suggestions that helped us to improve the paper.

                  \end{document}